\documentclass[11pt,a4paper]{scrartcl}
\usepackage[english]{babel}
\usepackage[utf8]{inputenc}
\usepackage[T1]{fontenc}
\usepackage[hidelinks]{hyperref}
\usepackage{subcaption}
\usepackage{paralist} \usepackage{float} \usepackage[ruled]{algorithm2e}
\SetKwInOut{Input}{Input}
\SetKwInOut{Output}{Output}
\usepackage{graphicx} \usepackage{tikz, pgfplots}
\pgfplotsset{compat=1.13}
\usepackage{mathrsfs} \usepackage{mathtools, amssymb, amsthm}
\usepackage{bm} 

\numberwithin{equation}{section}
\numberwithin{table}{section}
\numberwithin{figure}{section}

\newtheorem{theorem}{Theorem}[section]
\newtheorem{lemma}[theorem]{Lemma}
\newtheorem{remark}[theorem]{Remark}

\newtheorem{definition}[theorem]{Definition}

\newtheorem{corollary}[theorem]{Corollary}

\renewcommand{\epsilon}{\varepsilon}
\newcommand{\transp}{^\mathsf T}
\newcommand{\herm}{^\mathsf H}
\newcommand{\inv}{^{-1}}
\newcommand{\m}{_{(x)}}
\newcommand{\f}{\bm f}
\newcommand{\fhat}{\bm{\hat f}}
\newcommand{\ftilde}{\bm{\tilde f}}
\newcommand{\fht}{\bm{\tilde{f}} } \newcommand{\F}{\bm F}
\newcommand{\C}{\mathbb C}
\newcommand{\R}{\mathbb R}
\newcommand{\T}{{\mathbb T^d}}
\newcommand{\e}{\mathrm e}
\renewcommand{\S}{{\mathbb S^2}}

\DeclareMathOperator*{\diag}{diag}
\DeclareMathOperator*{\dist}{dist}
\DeclareMathOperator*{\trace}{trace}

\author{Felix Bartel\footnotemark[1] \and Ralf Hielscher\footnotemark[2] \and Daniel Potts\footnotemark[3]}
\title{Fast Cross-validation in Harmonic Approximation}

\date{} 

\begin{document}

\maketitle

\begin{abstract}
\noindent
Finding a good regularization parameter for Tikhonov regularization problems
is a though yet often asked question.  One approach is to use leave-one-out
cross-validation scores to indicate the goodness of fit.  This utilizes only
the noisy function values but, on the downside, comes with a high
computational cost. In this paper we present a general approach to shift the
main computations from the function in question to the node distribution and,
making use of FFT and FFT-like algorithms, even reduce this cost tremendously
to the cost of the Tikhonov regularization problem itself. We apply this
technique in different settings on the torus, the unit interval, and the
two-dimensional sphere. Given that the sampling points satisfy a quadrature
rule our algorithm computes the cross-validations scores in floating-point precision. In the cases of arbitrarily scattered
nodes we propose an approximating algorithm with the same complexity.
Numerical experiments indicate the applicability of our algorithms.

\noindent {\textit{Keywords and phrases:}}
cross-validation, regularization, fast evaluation of cross-validation score, discrete Fourier transforms, spherical Fourier transform, NFFT

\medskip

{\small \noindent {\textit{2010 AMS Mathematics Subject Classification:}}
\text{
65Txx, 65F22, 42A10. }
}
\end{abstract}
\footnotetext[1]{
Chemnitz University of Technology, Faculty of Mathematics, 09107 Chemnitz, Germany,\\
  felix.bartel@mathematik.tu-chemnitz.de, Phone:+49-371-531-35409, Fax:+49-371-531-835409
  }

\footnotetext[2]{
  Chemnitz University of Technology, Faculty of Mathematics, 09107 Chemnitz, Germany,\\
  hielscher@mathematik.tu-chemnitz.de, Phone:+49-371-531-38556, Fax:+49-371-531-838556
}
\footnotetext[3]{
  Chemnitz University of Technology, Faculty of Mathematics, 09107 Chemnitz, Germany,\\
  potts@mathematik.tu-chemnitz.de, Phone:+49-371-531-32150, Fax:+49-371-531-832150
}

\section{Introduction}

Estimating a good regularization parameter is a frequent problem in
approximation, statistics, and inverse problems. In this paper we restrict
ourselves to the basic example of approximating a function from discrete
function values. To make the setting concrete we fix for a finite index set
$\mathcal I$ a family of  basis functions $\varphi_{\bm n} \colon X \to \C$,
$\bm n \in \mathcal I$ on a domain $X \subset \R^{d}$.  Given a finite set of
nodes $\mathcal X \subset X$ and the corresponding Fourier matrix
\begin{equation*}
  \F
  = \F_{\mathcal X,\mathcal I}
  = (\varphi_{\bm n}(x))_{x \in \mathcal X, \bm n \in \mathcal I},
\end{equation*}
we consider the problem of recovering Fourier coefficients
$\fhat = (\hat f_{\bm n})_{\bm n \in \mathcal I} \in \C^{|\mathcal I|}$ from
noisy data
\begin{equation*}
  \f = (f_x)_{x \in\mathcal X} = \F\fhat+\bm\epsilon \in \C^{|\mathcal X|},
\end{equation*}
where $\bm \epsilon$ is zero mean Gaussian noise. More specifically, we look
for minimizers of the Tikhonov functional
\begin{equation}
  \label{eq:tikhonov}
  J_\lambda\left(\fht\right)
  = \left\|\F\fht-\f\right\|_{\bm W}^2+\lambda\left\|\fht\right\|_{\bm{\hat W}}^2
\end{equation}
and ask for the optimal regularization parameter $\lambda > 0$, where $\|\bm f\|_{\bm W}^2={\bm f}\herm\bm W \bm f$ and $\|\fhat\|_{\bm{\hat W}}^2 = \fhat\herm\bm{\hat W}\fhat$ for the strictly positive diagonal weight matrices $\bm W\in \mathbb R^{|\mathcal X|\times |\mathcal X|}$ and $\bm{\hat W}\in \mathbb R^{|\mathcal I|\times |\mathcal I|}$ in space, respectively frequency domain.

Because of its importance to many practical problems there is a vast
literature on many different strategies to determine an optimal regularization
parameter $\lambda$, e.g. \cite{DPR10,CY10,GoHeWa79,PeSlTk14}.
The idea of so called cross-validation methods is to
divide the set of nodes into a subset used to compute the approximation and a
subset used for validating the goodness of fit. This procedure might be
repeated for different splittings and eventually results for a fixed
regularization parameter $\lambda$ in a cross-validation score. By minimizing
this score an ``optimal'' regularization parameter is found. This approach was
initially proposed by Golub, Heath and Wahba \cite{GoHeWa79} in the setting of
smoothing splines and since then has been applied to a wide range of parameter
estimating problems.

In this paper we consider ``leave-one-out'' cross-validation, i.e., for
fixed regularization parameter $\lambda$ and any node $x \in \mathcal X$ we
compute the minimizer $\bm{\tilde f}_{\lambda,(x)}$ of the functional
\eqref{eq:tikhonov} restricted to the set of nodes
$\mathcal X \setminus \{x\}$ and use
\begin{equation*}
  P(\lambda)
  = \sum_{x \in \mathcal X} \left|\left[\F \bm{\tilde f}_{\lambda,(x)}\right]_x-f_x\right|^2
\end{equation*}
as cross-validation score.
A drawback with a purly data-driven regularization method like this is, that there is no guarantee for a good approximation of the solution of the regularized problem as formulated in the Bakushinski\u{\i} veto, \cite{Bak84}.
Another difficulty of this approach is its numerical
complexity. Indeed, computing $P(\lambda)$ for a single value of $\lambda$
requires solving the minimization problem \eqref{eq:tikhonov}
$\left| \mathcal X\right|$ times, which is too expensive for most
applications.  For spline interpolation on the interval or in higher
dimensional domains different algorithms have been proposed to lower the
computational costs. Those include Monte Carlo approximations \cite{DeGi91},
matrix decomposition methods \cite{Wei07,RoWiBu08} and Krylow space methods
\cite{LuHoAn10}.

For the specific setting of Fourier approximation on the torus $\T$ at regular
lattice points a fast algorithm has been proposed by Tasche and Weyrich
\cite{TaWe96} which requires to solve the minimization problem only once for
each regularization parameter $\lambda$. The idea of this paper is to
generalize the approach in \cite{TaWe96} to arbitrary sampling nodes and other
domains like the unit interval or the two-dimensional sphere.

The paper is organized as follows. In the second chapter we introduce the
necessary notations and prove in Theorem~\ref{theorem:simp1} a representation
of the cross-validation score that depends only on one solution of
\eqref{eq:tikhonov}, but includes the diagonal entries $h_{x,x}$ of the so
called hat matrix
$\bm H = \F (\F\herm \bm W \F + \lambda \bm{\hat W})\inv \F\herm
\bm W$. The efficient approximate computation of those diagonal entries for
different settings is subject to the remaining chapters.

Our essential requirement for the exact fast computation of the diagonal
entries $h_{x,x}$ is that the nodes $\mathcal X$ together with the weights
$\bm W$ form an exact quadrature rule. This requirement is satisfied for
regular tensor product grids and rank-1 lattices
on the $d$-dimensional torus, Chebyshev nodes on the interval $[-1,1]$, and,
e.g., Gauss Legendre nodes on the two-dimensional unit sphere $\mathbb S^{2}$.
The corresponding formulae for the diagonal entries $h_{x,x}$ are given in the
Theorems~\ref{theorem:simplification_torus}, \ref{theorem:i_diagonals}, and
\ref{theorem:s2_diagonals}. Together with fast Fourier algorithms on the torus
\cite{KeKuPo09,nfft3}, for rank-1 lattices \cite{kaemmererdiss,Ka_LFFT}, on the interval \cite{fepo02,nfft3}, and on the
sphere \cite{KeKuPo06b,nfft3} this allows the efficient evaluation of the cross-validation
 score $P(\lambda)$ with a numerical complexity close to
$\mathcal O(|\mathcal I| + |\mathcal X|)$. Numerical examples for all these
settings illustrate our findings.

In the case that no exact quadrature rule is known for the given interpolation
nodes we suggest approximating them by the volume of the corresponding Voronoi
cells. Our numerical tests in Section~\ref{subs:torus_app},
\ref{sec:appr-quadr}, and \ref{sec:appr-quadr-S} indicate a good approximation
of the true cross-validation score, which is much more expensive to compute.  The
\textsc{Matlab} code of our algorithm as well as for all numerical experiments
can be found on the GitHub repository
\texttt{https://github.com/felixbartel/fcv}.

\section{Cross-validation}
\label{sec:loocv}

Lets start this section by reminding that the minimizer of the Tikhonov
functional \eqref{eq:tikhonov} can be given explicitly.

\begin{lemma}
The unique Tikhonov minimizer of \eqref{eq:tikhonov} is
\begin{equation}\label{eq:tikhonov_minimizer}
  \fht_{\lambda} = \left(\F\herm \bm W \F+\lambda\bm{\hat W}\right)\inv \F\herm \bm W \f.
\end{equation}
\end{lemma}

\begin{proof}
  We look for stationary points by calculating the roots of the gradient of
  $J_\lambda$
\begin{equation*}
  \nabla J_\lambda\left(\fht_{\lambda}\right) =
  2\F\herm \bm W\F\fht_{\lambda}-2\F\herm \bm W \f + 2\lambda \bm{\hat W} \fht_{\lambda} \overset != \bm 0.
\end{equation*}
Because $\F\herm\F$ is positive semidefinite, $\bm W,\bm{\hat W}$, and
$\lambda$ are strictly positive we find that
$\F\herm \bm W\F+\lambda\bm{\hat W}$ is positive definite.  In particular it
is invertible such that the stationary point can be written as
$\fht_{\lambda} = (\F\herm \bm W\F+\lambda\bm{\hat W})\inv\F\herm \bm W\f$.  Using the
positive definiteness we see that $\fht_{\lambda}$ fulfills the required minimizing
property.
\end{proof}

\begin{remark}
  Since for random nodes $\mathcal X$ and many important function
    systems the matrix $\bm F\herm\bm W\bm F$ is invertible with probability
    one, see e.g. \cite{BaGr03}, we may relax the assumption on the frequency
    weights $\hat w_n$ to be only non negative. Especially, the zero order
    frequency weight is often set to zero to avoid penalizing constant
    functions. \end{remark}

As the leave-one-out cross-validation score depends on solving
\eqref{eq:tikhonov_minimizer} for sets of nodes of the form
$\mathcal X \setminus \{x\}$ we introduce the following notations for omitting
a single node $x \in \mathcal X$. For $x \in \mathcal X$ and
$\f \in \C^{|\mathcal X|}$ we denote by
\begin{equation*}
  \f_{(x)} = (f_{y})_{y \in \mathcal X \setminus \{x\}} \in \C^{|\mathcal X|-1}
\end{equation*}
the vector of function values $\f$ with one node $x \in \mathcal X$
omitted. Accordingly, we denote by
\begin{equation*}
  \bm F_{(x)}
  = (\varphi_{n}(y))_{y \in \mathcal X\setminus \{x\}, n \in \mathcal I}
  \in\C^{(|\mathcal X|-1)\times|\mathcal I|}
\end{equation*}
the Fourier matrix $\bm F$ with the row corresponding $x \in \mathcal X$
omitted and by
\begin{equation*}
  \bm W_{(x)}
  = (W_{y,y'})_{y,y' \in \mathcal X\setminus \{x\}}
  \in\C^{(|\mathcal X|-1)\times(|\mathcal X|-1)}
\end{equation*}
the restriction of the spatial weight matrix $\bm W$ to the set of nodes
$\mathcal X \setminus \{x\}$. With these notations the minimizer of the
Tikhonov functional \eqref{eq:tikhonov} reduced to the nodes
$\mathcal X \setminus \{x\}$ can be written as
\begin{equation}\label{eq:loo}
  \fht\m
  = \left(\F\m \herm \bm W\m\F\m+\lambda \bm{\hat W}\right)\inv \F\m\herm \bm W\m \f\m
  \in\C^{|\mathcal I|}.
\end{equation}

\begin{definition}
  The \emph{ordinary cross-validation score} for the Tikhonov functional
  \eqref{eq:tikhonov} is defined as
  \begin{equation}\label{eq:cv}
    P(\lambda) = \sum_{x\in\mathcal X} \left|\left[\F\fht\m\right]_x-f_x\right|^2
  \end{equation}
  where $\fht\m$ is defined by \eqref{eq:loo} and $[\F\fht\m]_x$ denotes the entry of $\F\fht\m$ corresponding to the node $x\in\mathcal X$.
\end{definition}

Interpreting \eqref{eq:cv}, we are comparing the predicted or smoothened value
$[\F\fht\m]_{\bm x}$ with the noisy data $f_{\bm x}$ for each node.
They intuitively differ more in the case of under- or oversmoothing.  So its
minimum is a candidate for the smoothing parameter $\lambda$.
Indeed, probabilistic optimality results of the form
$$
  \frac{\|\bm f-\bm F\bm{\tilde f}_{\lambda^\star}\|_2^2}{\inf_{\lambda\ge 0}\|\bm f-\bm F\bm{\tilde f}_{\lambda}\|_2^2}
  \to 1
$$
in probability for $\lambda^\star$ being the minimum of the cross-validation score have been shown for $|\mathcal X|\to\infty$ in \cite{Li86} under the assumption of homogenious noise, i.e., $\epsilon_x$ has the same variance for all $x\in\mathcal X$.
In Theorem~3.5 of \cite{Gu13}, weights were incoporated and the condition was generalized to the case where all $\sqrt{w_x\epsilon_x}$ have a common variance for $x\in\mathcal X$.
These optimality results are with respect to a LOSS function similar to the data-fitting term in the Tikhonov functional \eqref{eq:tikhonov}.
However, in our numerical experiments, we will have a look at the unweighted $L_2$-error as we use the weights $w_x$ for numerical purposes and want to analyze how this disstorts the results.
Unflattering is
the fact that, for a single regularization parameter $\lambda$,
the direct computation of the ordinary cross-validation score requires to solve
$|\mathcal X|$ times the normal equation \eqref{eq:tikhonov_minimizer}.

Our first goal is to relate the solution of the reduced problem \eqref{eq:loo}
to the solution of the full problem \eqref{eq:tikhonov_minimizer}. To this end
we define the matrices
\begin{equation}
  \label{eq:2}
    \begin{split}
    \bm A &= \F\herm \bm W\F+\lambda\bm{\hat W}, \\
    \bm A\m &= \F\m\herm \bm W\m\F\m+\lambda\bm{\hat W}
  \end{split}
\end{equation}
which are decisive for the computation of $\fht$ and $\fht_{(x)}$,
respectively, and show the following relationship between their inverse,
cf. \cite{GoHeWa79}.
\begin{lemma}
  \label{lemma:Sherman-Morrison}
  Let $\bm A$ and $\bm A\m$ be defined as in \eqref{eq:2} and
  \begin{equation*}
    \F_{x,:} = (\varphi_{n}(x))_{n \in \mathcal I}
    \in\C^{1\times|\mathcal I|}
  \end{equation*}
  denote the row of the matrix $\F$ which corresponds to the node $x \in \mathcal
  X$. Then we have
  \begin{equation}
    \label{eq:Ainv}
    \bm A\m\inv
    = \bm A\inv +
    \frac{\bm A\inv w_x \F_{x,:}\herm \F_{x,:} \bm A\inv}
    {1-w_x \F_{x,:} \bm A\inv \F_{x,:}\herm}.
  \end{equation}
\end{lemma}

\begin{proof}
  The assertion of the lemma follows immediately by applying the Sherman-Morrison formula to
  \begin{equation*}
    \bm A\m
    = \bm A-w_x \F_{x,:}\herm \F_{x,:}.
  \end{equation*}
\end{proof}

Our next goal is to make the repetitive solving of the normal equation
\eqref{eq:loo} in \eqref{eq:cv} independent of the right-hand side $\f$. To
this end we define the so called \emph{hat matrix}
\begin{equation}
  \label{eq:hat}
  \bm H  \coloneqq \F \bm A \inv\F\herm \bm W
  = \F\left(\F\herm \bm W\F+\lambda\bm{\hat W}\right)\inv\F\herm \bm W
\end{equation}
which when applied to a data vector $\f$ solves the normal equation
\eqref{eq:tikhonov_minimizer} and evaluates the resulting function at the
nodes $\mathcal X$. The next lemma is a generalization of \cite[equation
2.2]{GoHeWa79} and \cite[equation 4.2.9]{Wa90}, and shows that for the
computation of \eqref{eq:cv}, with given diagonal entries
$h_{x,x}$ of the hat matrix $\bm H$, it is sufficient to solve the normal equation
\eqref{eq:tikhonov_minimizer} with respect to the data vector $\f$ only once.

\begin{theorem}
  \label{theorem:simp1}
  The ordinary cross-validation score \eqref{eq:cv} can be written as
  \begin{equation}\label{eq:cv_simple}
    P(\lambda)
    = \sum_{x \in \mathcal X} \frac {[\bm H\f-\f]_{x}^{2}}{(1-h_{x,x})^{2}}
\end{equation}
  with $h_{x,x}$, $x \in \mathcal X$ being the diagonal entries of the hat
  matrix $\bm H$ defined in \eqref{eq:hat}.
\end{theorem}

\begin{proof}
  Let $\bm b = \F\herm\bm W\f$. Then
\begin{equation*}
    \fht\m
    = \bm A_{(x)}\inv \F_{(x)}\herm \bm W_{(x)} \f_{(x)}
    = \bm A_{(x)}\inv \left(\bm b - \F_{x,:}\herm w_{x}f_{x}\right).
  \end{equation*}
  Next we apply Lemma~\ref{lemma:Sherman-Morrison} and observe that the
  denominator in \eqref{eq:Ainv} can be expressed in terms of the diagonal
  entries $h_{x,x}$ of the hat matrix $\bm H$:
  \begin{align*}
    \fht\m
    &= \bm A\m \inv \left(\bm b - \F_{x,:} \herm w_{x}f_{x}\right)\\
    &= \left(\bm A\inv + \frac{w_x\bm A\inv \F_{x,:}\herm \F_{x,:}
      \bm A\inv}{1-h_{x,x}}\right)\left(\bm b-\F_{x,:}\herm w_xf_x\right) \\
    &= \fht + \frac{w_x\bm A\inv\F_{x,:}\herm f_x(h_{x,x}-1)
      + w_x\bm A\inv \F_{x,:}\herm \F_{x,:} \fht
      - w_x\bm A\inv\F_{x,:}\herm h_{x,x}f_x}{1-h_{x,x}} \\
    &= \fht+w_x\bm A\inv\F_{x,:}\herm \frac{\F_{x,:}\fht-f_x}{1-h_{x,x}}\\
    &= \fht+w_x\bm A\inv\F_{x,:}\herm \frac{[\bm H\f-\f]_x}{1-h_{x,x}}.
  \end{align*}

  Multiplying with $\F$ from the left-hand side and subtracting $f_x$
  results in
  \begin{align*}
    \left[\F\fht\m\right]_x-f_x
    &=\F_{x,:}\fht\m - f_x\\
    &= \F_{x,:}\fht
      + w_x\F_{x,:} \bm A\inv\F_{x,:}\herm \frac{[\bm H\f-\f]_x}{1-h_{x,x}} - f_x\\
    &= \left[\F\fht\right]_x+h_{x,x}\frac{[\bm H\f-\f]_x}{1-h_{x,x}} - [\f]_x\\
    &= [\bm H\f]_x+\frac{[\bm H\f-\f]_x}{1-h_{x,x}}+[\f-\bm H\f]_x - [\f]_x\\
    &= \frac{[\bm H\f-\f]_x}{1-h_{x,x}}\end{align*}
  and hence each summand in \eqref{eq:cv} is equal to the corresponding
  summand in \eqref{eq:cv_simple}.
\end{proof}

\begin{remark}
  \label{remark:cv_canonical_calculation}
  According to Theorem~\ref{theorem:simp1} the ordinary cross-validation score
  is nothing more than the weighted norm of the residue
  \begin{equation*}
    \bm r = \F\fht-\f = \bm H\f-\f.
  \end{equation*}
  Although this means that the normal equation \eqref{eq:tikhonov_minimizer}
  has to be solved only once with respect to the data vector $\f$ the most
  expensive part remains, namely the computation of the diagonal entries
  \begin{equation*}
    h_{x,x} = w_{x} \F_{x,:} \bm A\inv \F_{x,:}\herm
  \end{equation*}
  for $x \in \mathcal X$, which again requires repetitive solving of the normal equation.
\end{remark}

Replacing the diagonal entries $h_{x,x}$ with their mean value
\begin{equation*}
  h
  = \frac{1}{|\mathcal X|} \sum_{x \in \mathcal X} h_{x,x}
  = \frac{1}{|\mathcal X|} \trace \bm H
\end{equation*}
we obtain the so called generalized cross-validation score,
cf.\,\cite[section 4.3]{Wa90}.

\begin{definition}
  \label{def:generalizedCV}
  The \emph{generalized cross-validation score} is defined as
  \begin{equation*}\label{eq:gcv}
    V(\lambda)
    = \sum_{x \in \mathcal X} \frac {[\bm H\f-\f]_{x}^{2}}{(1-h)^{2}}
    = \left(\frac{|\mathcal X| \left\| \bm H\f-\f\right\|_2} {\trace(\bm I-\bm H)} \right)^{2}.
  \end{equation*}
\end{definition}

Obviously, if all diagonal entries $h_{x,x}$ of $\bm H$ coincide we have
$P(\lambda) = V(\lambda)$.

\begin{lemma}
  \label{lemma:h<1}
  The diagonal elements $h_{x,x}$ of the hat matrix $\bm H$ satisfy
  \begin{equation*}
   h_{x,x}<1
  \end{equation*}
  for all $\lambda > 0$ and $x \in \mathcal X$.
\end{lemma}

\begin{proof}
  Since $\F_{(x)}\herm \bm W_{(x)} \F_{(x)}$ is positive semidefinite and
  $\lambda \bm{\hat W}$ is strictly positive definite we see that
  $\bm A\m=\F\m\herm \bm W\m\F\m+\lambda\bm{\hat W}$ is invertible. On the
  other hand we know by the Sherman-Morrison formula that
  $\bm A\m = \bm A - w_x\F_{x,:}\herm\F_{x,:}$ is invertible if and only if
  $w_x\F_{x,:}A\inv\F_{x,:}\herm \neq 1$ .  Therefore
  \begin{equation*}
    h_{x,x} = w_x\F_{x,:}\bm A\inv\F_{x,:}\herm \neq 1.
  \end{equation*}

  Since the minimizer $\ftilde_{\lambda}$ of \eqref{eq:tikhonov} converges to the
  zero vector as $\lambda\to\infty$, we obtain for $\f=\bm e_x$ and $\lambda\to\infty$
  \begin{equation*}
    h_{x,x} = \left[\F\ftilde_{\lambda}\right]_x \to 0.
  \end{equation*}
  Together with the fact that the diagonal entries $h_{x,x}$ depend
  continuously on $\lambda$ this proves the assertion.
\end{proof}

\subsection{Algorithm to compute the ordinary or generalized cross-validation score}

Concluding the previous statements we end up with a sheme to compute the cross-validation scores.

\begin{figure}[H]
\centering
\begin{minipage}{0.9\textwidth}
\begin{algorithm}[H]
  \label{algo:cv}
  \caption{generic computation of the cross-validation scores}

  \textbf{Input:}

  \quad nodes $\mathcal X$

  \quad spatial weights $\bm W = \diag(w_x)_{x\in\mathcal X}\in \mathbb R^{|\mathcal X| \times |\mathcal X|}  $

  \quad Fourier weights $\bm{\hat W}\in
  \mathbb R^{|\mathcal I| \times |\mathcal I|} $

  \quad function values $\f=(f_x)_{x\in\mathcal X}$

  \quad regularization parameter $\lambda$

  \medskip

  \textbf{Output:}

  \quad ordinary cross-validation score $P(\lambda)$

  \quad generalized cross-validation score $V(\lambda)$

  \medskip

  \textbf{1.} Compute $\ftilde \coloneqq \bm H \f = \F \bm A\inv \F\herm \bm W \f$,  where $\bm A$ is given in \eqref{eq:2}. \\
  \textbf{2.} Compute $ h_{x,x} \coloneqq w_{x} \F_{x,:} \bm A\inv \F_{x,:}\herm$ for $x \in \mathcal X$ and $\displaystyle h \coloneqq \frac 1{|\mathcal X|} \sum_{x\in\mathcal X} h_{x,x}$.\\
  \textbf{3.} Evaluate
  $\displaystyle P(\lambda) \coloneqq \sum_{x\in\mathcal X}\frac{|\tilde
    f_{x}-f_{x}|^{2}}{1-h_{x,x}}$ and
  $\displaystyle V(\lambda) \coloneqq \sum_{x\in\mathcal X}\frac{|\tilde
    f_{x}-f_{x}|^{2}}{1-h}$.
\end{algorithm}
\end{minipage}
\end{figure}

\begin{remark}
For computing the Tikhonov-minimizer of \eqref{eq:tikhonov} one can use the LSQR method for numerical stability.
This can be accomplished with the coefficient matrix
$$
  \bm M = \begin{pmatrix} \bm W^{1/2}\F \\ \sqrt\lambda\bm{\hat W}^{1/2} \end{pmatrix}
$$
and the right-hand side
$$
  \bm b = \begin{pmatrix}\bm W^{1/2} \f \\ \bm 0 \end{pmatrix},
$$
where $\bm 0$ is a column vector containing $|\mathcal I|$ zeros.
The resulting system of equations $(\bm M\herm \bm M)\inv\bm M\herm\bm b$, which the LSQR method solves, is equivalent to \eqref{eq:tikhonov_minimizer}.
\end{remark}

The computationally most expensive part of Algorithm~\ref{algo:cv} is the
computation of the values $\ftilde$ and $h_{x,x}$ for all $x \in \mathcal X$.
In the subsequent sections we discuss some specific settings to speed up the
process and propose an approximation of the ordinary and the generalized
cross-validation score in more general cases.

\section{Cross-validation on the torus}
\label{sec:t}

In this section, we seek for fast algorithms to compute the cross-validation
score on the $d$-dimensional torus $\T$ with respect to the Fourier basis
$\{\e^{2\pi\mathrm i\bm n\cdot\bm x}\}_{\bm n\in\mathbb Z^d}$ in $L_2(\T)$.
With this setting the Fourier matrix $\F$ becomes
\begin{equation}\label{eq:FM}
 \F = \F_{\mathcal X,\mathcal I}
  = \left(\e^{2\pi\mathrm i\bm n\cdot \bm x}\right)_{\bm x\in \mathcal X,\bm n\in \mathcal I}
\end{equation}
for a finite node set $\mathcal X \subset\T$, a finite multi-index set
$\mathcal I \subset \mathbb Z^d$ and $\bm n\cdot\bm x$ the Euclidean inner
product.  So $\mathcal I$ determines all possible frequencies and $\mathcal X$
the nodes of the transform. For the specific case of equispaced nodes
$\mathcal X$ fast algorithms have been reported in \cite{TaWe96}. In fact, our approach in this section can be seen as a
generalization of \cite{TaWe96} to more general sampling sets and leads
to the same algorithm for equispaced data.

Our central goal is to find a simpler expression for the diagonal entries
of the hat matrix
$\bm H = \F(\F\herm \bm W\F+\lambda\bm{\hat W})\inv\F\herm \bm W$ that allows
us to apply Theorem~\ref{theorem:simp1} efficiently. The idea is to
choose $\bm W$ such that $\F\herm \bm W\F$ has diagonal form because the
inverse of $\bm A = \F\herm\bm W\F+\lambda\bm{\hat W}$ could then be
calculated entry-wise.

\subsection{Exact Quadrature}
\label{sec:exact-quadrature-torus}

The first approach is to use quadrature rules. Because they are not limited to
the torus we define them for general measure spaces so we can make use of them
in subsequent sections.

\begin{definition}
  \label{def:quadrature}
  Let $(\mathcal M,\mu)$ be a measure space and
  $\mathcal P \subset L_{1}(\mathcal M)$ a set of integrable functions. We
  call a set of nodes $\mathcal X\subset \mathcal M$ and a list of weights
  $\bm W = \diag(w_{\bm x})_{\bm x\in\mathcal X}$ an \emph{exact quadrature rule for
    $\mathcal P$}, if for all $f\in\mathcal P$ we have
  \begin{equation*}
    Q_{\mathcal X,\bm W}f \coloneqq \sum_{\bm x\in\mathcal X} w_{\bm x} f(\bm x)
    = \int_{\mathcal M} f(\bm x)\;\mathrm d\mu(\bm x).
  \end{equation*}
\end{definition}

For the tours we obtain the following

\begin{theorem}
  \label{theorem:simplification_torus}
  Let $\mathcal I\subset \mathbb Z^d$ be a finite multi-index set with Fourier
  weights $\bm{\hat W}=\diag(\hat w_{\bm n})_{\bm n\in \mathcal I}$,
  $\mathcal X\subset \T$ a set of nodes with $\bm W$ their corresponding
  quadrature weights such that $(\mathcal X,\bm W)$ forms a quadrature rule
  which is exact for all trigonometric polynomials $\e^{2\pi \mathrm i \bm n \cdot}$
  with frequencies $\bm n$ in the difference set
  $\mathcal D(\mathcal I)\coloneqq\{\bm n_1-\bm n_2:\bm n_1,\bm n_2\in
  \mathcal I\}$. Then
  \begin{compactenum}[(i)]
  \item
  the inverse of $\bm A$, given in \eqref{eq:2}, where $\F$ is the Fourier matrix \eqref{eq:FM} on $\T$ is
  \begin{equation*}
    \bm A\inv
    = \left(\F\herm\bm W\F+\lambda\bm{\hat W}\right)\inv
    = \diag\left(\frac 1{1+\lambda\hat w_{\bm n}}\right)_{\bm n\in\mathcal I},
  \end{equation*}
  \item
  the diagonal entry corresponding to the node $\bm x\in\mathcal X$ of the hat matrix $\bm H = \F\bm A\inv\F\herm \bm W$ becomes
  \begin{equation}
    \label{eq:diagonals_simplification}
    h_{\bm x,\bm x}
    = w_{\bm x}\sum_{\bm n\in \mathcal I} \frac 1{1+\lambda \hat w_{\bm n}}.
  \end{equation}
  \end{compactenum}
\end{theorem}

\begin{proof}
  Since the product of two exponential functions supported on the frequency
  set $\mathcal I$ has only frequencies in $\mathcal D(\mathcal I)$, where the
  quadrature nodes and weights are exact, we have
  \begin{equation*}
    \left[\F\herm\bm W\F\right]_{\bm n_1,\bm n_2}
    = \sum_{\bm x\in\mathcal X} w_{\bm x} \overline{\e^{2\pi\mathrm i\bm n_1\bm{x}}}\e^{2\pi\mathrm i\bm n_2\bm{x}}
    = \int_\T \overline{\e^{2\pi\mathrm i\bm n_1\bm x}}\e^{2\pi\mathrm i\bm n_2\bm x} \;\mathrm d\bm x.
  \end{equation*}
  and, hence,
  \begin{equation*}
    \F\herm\bm W\F+\lambda\bm{\hat W}
    = \diag\left(1+\lambda\hat w_{\bm n}\right)_{\bm n\in\mathcal I}.
  \end{equation*}
  This implies (i). For (ii) we compute the diagonal entries of $\bm H$ as
  \begin{equation*}
    \label{eq:t_diagonals}
    h_{\bm x,\bm x} = w_{\bm x}\sum_{\bm n\in \mathcal I} \frac 1{1+\lambda \hat w_{\bm n}} \e^{2\pi\mathrm i\bm n\bm x}\e^{-2\pi\mathrm i\bm n\bm x}
    = w_{\bm x}\sum_{\bm n\in \mathcal I} \frac 1{1+\lambda \hat w_{\bm n}}.
  \end{equation*}
\end{proof}

\begin{corollary}
\label{corollary:t_fast_exact}
With the prerequisites of Theorem~\ref{theorem:simplification_torus} we can compute $P(\lambda)$ and $V(\lambda)$ by Algorithm~\ref{algo:cv} in the same complexity as multiplying a vector with $\F$ or $\F\herm$ for a fixed regularization parameter $\lambda$.
\end{corollary}

\begin{remark}\label{remark:problemalteringweights}
  Theorem~\ref{theorem:simplification_torus} requires $(\mathcal X, \bm W)$ to form an exact quadrature rule which directly alters the weighting of the data-fitting term in the underlying Tikhonov functional~\eqref{eq:tikhonov}.
  Preferebly, one wants to have the weight $w_x$ proportional to the level of noise $\epsilon_x$ for every node $x\in\mathcal X$, cf.\,\cite{Re67}.
  In most of the following examples the weights as well as the Gaussian noise are uniform, thus fulfulling this property.
\end{remark}

Case studies for specific exact quadrature rules on the torus are discussed in the following two subsections.

\subsection{Equispaced Nodes}
\label{sec:torus_equispaced}

The simplest example of quadrature on the torus $\mathbb T^{d}$ is Gauss
quadrature which consists of $N^d$ equispaced nodes
\begin{equation*}
  \mathcal X = \left\{\frac{\bm m}N\in\T:\bm m\in\mathbb Z^d\cap\prod_{t=1}^d[0,N)\right\}
\end{equation*}
with uniform weights $w_{\bm x} = N^{-d}$.  The resulting quadrature formula
is exact for all trigonometric polynomials supported on
$\mathcal I_{2N} \coloneqq \mathbb Z^d\cap\prod_{t=1}^d[-N,N) = \mathcal
D(\mathcal I_N)$.  Thus we can apply Theorem~\ref{theorem:simplification_torus} for $\mathcal X$ and
$\mathcal I = \mathcal I_N$.  The corresponding Fourier matrix
$\F=\F_{\mathcal X,\mathcal I_N}$ describes the ordinary discrete Fourier
transform for which the matrix-vector product can be computed in
$\mathcal O(N^d\log N)$ with the fast Fourier transform.

For $d=1,2$ our algorithm coincides with the algorithm proposed in
\cite{TaWe96} with the only difference that the authors evaluated the data
fitting term in the frequency domain and used specific Fourier weights
$w_{n} = n^{4}$ as regularization term.

In order to illustrate our approach we chose as the test function $f$ the
peaks function from \textsc{Matlab}, which is a sum of translated and scaled
Gaussian bells. We evaluated this function on a grid of $1024\times 1024$
equispaced nodes $\mathcal X$ and corrupted the data by 10\% Gaussian noise
$\epsilon_{\bm x}$, i.e., we set
\begin{equation*}
  f_{\bm x} = f(\bm x)+\epsilon_{\bm x}
\end{equation*}
for all $\bm x\in\mathcal X$ as depicted in Figure~\ref{fig:equispaced_2d}, (a).

\begin{figure}
  \centering

  \begin{subfigure}[b]{0.25\textwidth}
    \includegraphics[width=\textwidth]{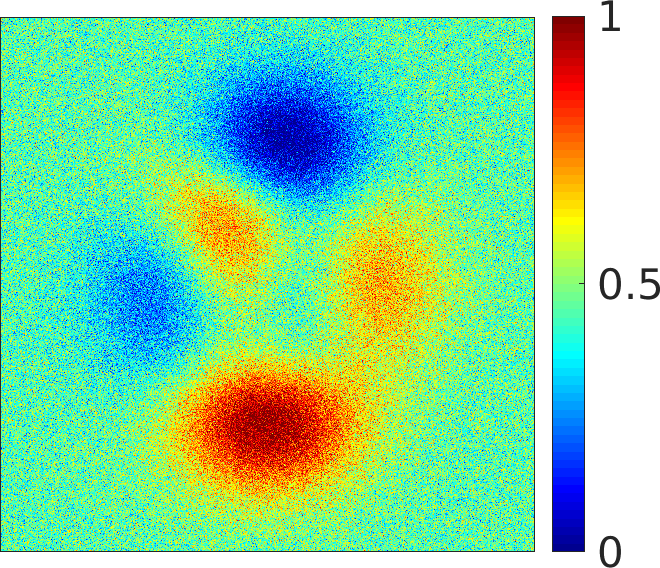}
    \\ \\
    \includegraphics[width=\linewidth]{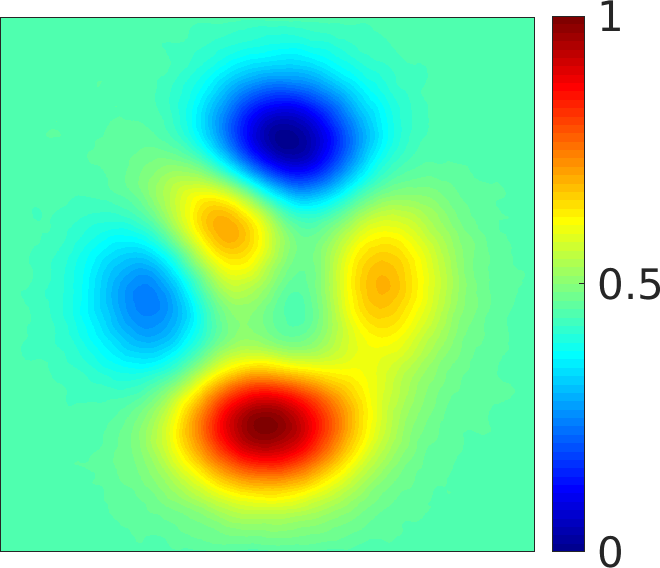}
    \caption{noisy input data and reconstruction}
  \end{subfigure}
  \quad
  \begin{subfigure}[b]{0.5\textwidth}
    \begin{tikzpicture}[font = \footnotesize]
      \begin{axis}[
        scale only axis,
        width             = 0.8\textwidth, height = 4cm,
        xlabel            = $\lambda$,
        xmode             = log,
        ymode             = log,
        axis y line*      = left,
        enlarge x limits  = {abs = 0},
        legend pos        = north east,
        legend cell align = left,
      ]
        \addplot[no marks,thick] table[x index = 0, y index = 1] {data/t_equispaced.dat};
        \addlegendentry{$\lVert \tilde \f_{\lambda} - \hat \f \rVert_{2}$};
        \addlegendimage{solid,orange,thick};
        \addlegendentry{$P(\lambda)$};
        \addplot [mark=*] coordinates {(4.2507352e-04,2.4250072e-03)};

      \end{axis}
      \begin{axis}[
        scale only axis,
        width             = 0.8\textwidth, height = 4cm,
        xmode             = log,
        ymode             = log,
        axis y line*      = right,
        axis x line       = none,
        axis line style   = {orange},
        yticklabel style  = {color = orange},
        enlarge x limits  = {abs = 0},
        ytick             = {1.05,1.055},
        yticklabels       = {$10^{4.021}$,$10^{4.023}$},
      ]
        \addplot[no marks,orange,thick] table[x index = 0, y index = 2] {data/t_equispaced.dat};
        \addplot [mark=*,color=orange] coordinates {(4.6799330e-04,1.0485906)};
      \end{axis}
    \end{tikzpicture}
    \caption{approximation error
      $\lVert \tilde \f_{\lambda} - \hat \f \rVert_{2}$ (black) and cross-validation
      scores (orange)}
  \end{subfigure}

  \caption{Approximation from two-dimensional equispaced data: Comparison
    of the ordinary cross-validation score $P(\lambda)$ and the approximation error.}
  \label{fig:equispaced_2d}
\end{figure}
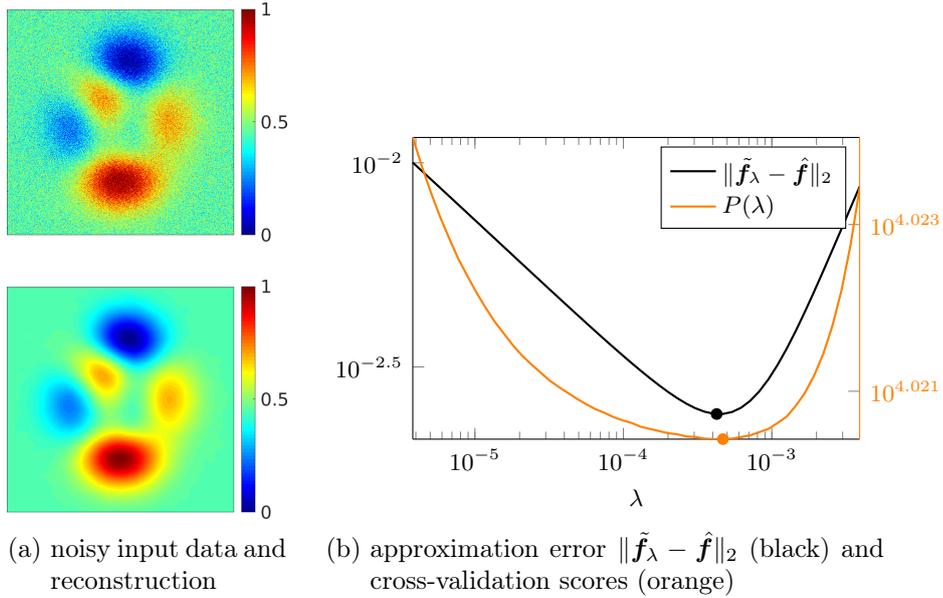

As regularization term we fixed isotropic Sobolev weights
$\hat w_{\bm n}=1+\|\bm n\|_2^s$ for $\bm n\in\mathcal I_N$ and $s=3$ in Fourier space,
which correspond to a function with $3$ derivatives in $L_2(\T)$.
Varying the regularization parameter $\lambda \in [2^{-18},2^{-8}]$ we
computed the Tikhonov minimizers $\ftilde_{\lambda}$ according to
\eqref{eq:tikhonov_minimizer}. We then applied Parseval to the original
$\fhat$ and $\ftilde_\lambda$ which is a byproduct from
Algorithm~\ref{algo:cv} to compute the $L_2(\T)$-approximation errors
as a function of the regularization parameter $\lambda$. The resulting curve
is depicted in Figure~\ref{fig:equispaced_2d}, (b).  Note that according to
\eqref{eq:diagonals_simplification} all diagonal entries of the hat matrix are
equal and, hence, the ordinary cross-validation score coincides with the
generalized cross-validation score. The reconstruction $\ftilde_{\lambda}$
with respect to the minimizer of the cross-validation score $P(\lambda)$ is
depicted in Figure~\ref{fig:equispaced_2d}, (a).

In Figure~\ref{fig:equispaced_2d}, (b) the actual $L_2(\T)$-approximation
error is compared to the cross-validation score $P(\lambda)$ computed
according to Algorithm~\ref{algo:cv} with use of the fast Fourier transform.
We observe that the minimizers of both functionals coincide surprisingly well.
For this numerical experiment the average running time for the evaluation of
$P(\lambda)$ for a single value of $\lambda$ was $0.06$ seconds with the fast
algorithm and more than $14$ hours for a direct implementation of
\eqref{eq:cv}.

\subsection{Rank-1 Lattices}
\label{sec:torus_r1l}

The approximation of high-dimensional multivariate periodic functions by
trigonometric polynomials using particular finite index sets $\mathcal I$ in
frequency domain is possible using special index sets \cite{Tem93, DuTeUl2018}
on the domain $\mathcal X$. The most efficient method uses samples along
rank-1 lattices and is based on a simple univariate FFT
\cite{KaPoVo13}. Rank-1 lattices are defined by
\begin{equation*}
  \mathcal X = \Lambda(\bm z,M)
  \coloneqq \left\{\bm x=\frac 1M(m\bm z\bmod M\bm 1)\in\T:m=0,\dots,M-1\right\}
\end{equation*}
where $M\bm 1=(M,\dots,M)\transp\in\mathbb Z^d$.  They are fully characterized
by the \emph{generating vector} $\bm z \in \mathbb Z^{d}$ and the \emph{lattice size} $M$.
There exist algorithms which, given a frequency index set $\mathcal I$ and
$M$, compute a generating vector $\bm z$ such that $\F\herm\bm W\F$ equals the
identity matrix for $\bm W = \diag(1/M)_{\bm x\in\mathcal X}$,
cf.~\cite{KaPoVo13, kaemmererdiss, PlPoStTa18}.  The advantage of rank-1
lattices is the variable index set $\mathcal I$ instead of the tensor-product
approach like in Section~\ref{sec:torus_equispaced}.  So depending on the
function we can adapt to different decay properties of the Fourier
coefficients.  Furthermore there exist fast algorithms which evaluate the
matrix-vector product with $\F$ or $\F\herm$ in
$\mathcal O(M\log M+d|\mathcal I|)$ using only one one-dimensional fast
Fourier transform.

To exemplify these ideas we looked at a sample function consisting of a tensor
product of $L_2(\T)$-normed B-splines of order two in seven dimensions,
i.e., $d=7$,
\begin{equation*}
  f(\bm x)
  = \prod_{j=1}^d B_2(x_j),\quad B_2(x) = 2\sqrt 3\left(\mathcal X_{[0,0.5)}x+\mathcal X_{[0.5,1)}(1-x)\right)
\end{equation*}
where $\mathcal X_A$ denotes the indicator function. The Fourier coefficients
of $f$ are
\begin{equation*}
  \hat f_{\bm n} = \prod_{j=1}^d \begin{cases}
    \sqrt{3/4} & : n_j = 0 \\
    \sqrt{3/4}\left(\frac{\sin(n_j\pi /2)}{n_j \pi /2}\right)^2\cos(n_j \pi) &
    :\text{otherwise.}
  \end{cases}
\end{equation*}
Therefore the Fourier coefficients decay like $\mathcal O(\prod_{j=1}^d n_j^{-2})$ and a candidate
for an index set $\mathcal I$ would be a $d$-dimensional hyperbolic cross
\begin{equation*}
  \mathcal I_N^{d,\mathrm{hc}} \coloneqq \left\{\bm n\in\mathbb Z^d : \prod_{j=1}^d \max(1,|n_j|)\le N\right\}.
\end{equation*}

In particular we used a radius of $N=16$, a reconstructing rank-1 lattice
$\mathcal X$ from \cite[table 6.2]{KaPoVo13} with $M=1\,105\,193$ nodes and set
the weights in Fourier space to
$\hat w_{\bm n} = \prod_{j=1}^d\max(|n_j|^2,1)$.

Applying Algorithm~\ref{algo:cv} to $f_{\bm x} = f(\bm x) + \varepsilon$,
$\bm x \in \mathcal X$, where $\varepsilon$ denotes 5\% Gaussian noise we
computed the cross-validation scores $P(\lambda) = V(\lambda)$ for
$\lambda\in[2^{-9},2^0]$. Again both scores coincide since the diagonal
entries \eqref{eq:diagonals_simplification} of the hat matrix are multiples of
the constant weights $w_{\bm x}$ in spatial domain. For the multiplications
with $\F$ and $\F\herm$ we made use of fast rank-1 lattices Fourier
transforms. A comparison of the actual $L_2(\T)$-approximation error with the
cross-validation score is plotted in Figure~\ref{img:r1l_6d}. We observe that
the optimal $\lambda$ with respect to the $L_2(\T)$-error and the $\lambda$
chosen by cross-validation are very close in this example.

\begin{figure}
  \centering

  \begin{tikzpicture}[font = \small]
    \begin{axis}[
      scale only axis,
      width             = 0.5\textwidth, height = 4cm,
      xlabel            = $\lambda$,
      xmode             = log,
      ymode             = log,
      axis y line*      = left,
      enlarge x limits  = {abs = 0},
      legend pos        = north east,
      legend cell align = left,
      ytick            = {0.17, 0.1514},
      ]
      \addplot[no marks,thick] table[x index = 0, y index = 1] {data/t_r1l.dat};
      \addlegendentry{$\lVert \tilde \f_{\lambda} - \hat \f \rVert_{2}$};
      \addlegendimage{solid,orange,thick};
      \addlegendentry{$P(\lambda)$};
      \addplot [mark=*] coordinates {( 2.7127847e-02,1.4738674e-01)};

    \end{axis}
    \begin{axis}[
      scale only axis,
      width             = 0.5\textwidth, height = 4cm,
      xmode             = log,
      axis y line*      = right,
      axis x line       = none,
      axis line style   = {orange},
      yticklabel style  = {color = orange},
      enlarge x limits  = {abs = 0},
      ytick            = {19.196,19.198}, yticklabels       = {$10^{8.3367}$,$10^{8.3372}$},
      ]
      \addplot[no marks,orange,thick] table[x index = 0, y index = 2] {data/t_r1l.dat};
      \addplot [mark=*,color=orange] coordinates {( 2.7127847e-02,1.9195005e+01)};
    \end{axis}
  \end{tikzpicture}

  \caption{Approximation in $\mathbb T^{7}$ from data at a rank-1 lattice: Comparison of the ordinary cross-validation score $P(\lambda)$ with the
    approximation error.}
  \label{img:r1l_6d}
\end{figure}
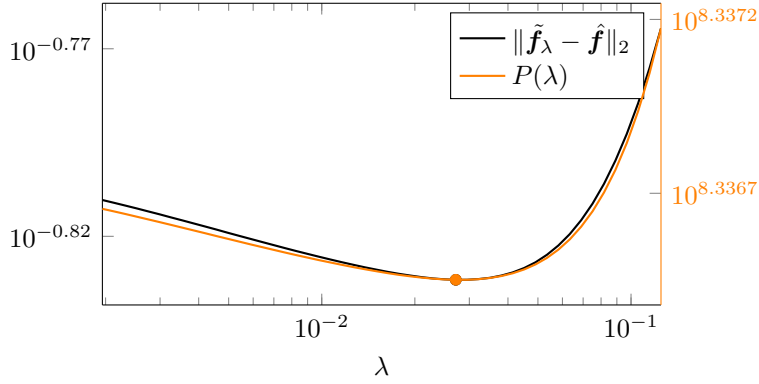

\subsection{Approximative quadrature}
\label{subs:torus_app}

In the case of scattered data approximation exact quadrature rules are
typically not available. In principle, one can compute exact quadrature rules
by determining the weighs $\bm W = \diag(w_{\bm x})_{\bm x\in\mathcal X}$ as a solution of
the linear system
\begin{equation*}\label{eq:quadrature_t}
  \bm F_{\mathcal X,\mathcal D(\mathcal I)}\herm \bm w
  = \left(\e^{2\pi\mathrm i\bm x\cdot\bm n}\right)_{\bm x\in\mathcal X,\bm n\in\mathcal D(\mathcal I)}\herm\bm w
  = \bm e_{\bm 0},
\end{equation*}
where $\bm e_{\bm 0}$ is the vector which only contains zeros, except in the
$\bm 0$-th position where it is equal to one. These weights can be
guaranteed to be non negative under certain conditions on the frequency index
set $\mathcal I$ and the mesh norm
\begin{equation*}
  \delta_{\mathcal X} \coloneqq \max_{\bm y\in\mathcal M} \min_{\bm x\in\mathcal X} |\bm y-\bm x|
\end{equation*}
of the nodes $\mathcal X$, cf.~\cite{FiMh11}. However, those conditions strongly
restrict the polynomial degree and do not guaranty the quadrature weights to
be non-oscillating. This may be problematic, since the weights directly alter
the problem \eqref{eq:tikhonov} we want to solve.

Our idea is to replace the exact quadrature weights $\bm W$ by approximative
weights derived from the Voronoi tessellation of the node set $\mathcal X$.

\begin{definition}\label{def:voronoi}
  Let $\mathcal M$ be a Riemannian manifold with a distance function
  $\dist(\cdot,\cdot)$.  For a set of nodes $\mathcal X\subset\mathcal M$ we define
  the \emph{Voronoi cell} $V_{\bm x}$ corresponding to $\bm x\in\mathcal X$ by
  \begin{equation*}
    V_{\bm x} \coloneqq \left\{\bm y\in\mathcal M :
      \dist(\bm x,\bm y) \le \dist(\bm{x^\prime},\bm y), \;\forall \bm{x^\prime}\in\mathcal X\right\}.
  \end{equation*}
  The \emph{Voronoi weight} $w_{\bm x}$ is the area of the Voronoi cell $V_{\bm x}$
  \begin{equation*}
    w_{\bm x} \coloneqq \int_{\mathcal M} \mathcal X_{V_{\bm x}}(\bm y) \;\mathrm d\bm y = \int_{V_{\bm x}}\mathrm d\bm y.
  \end{equation*}
\end{definition}

To emphasize the choice of the Voronoi weights as approximative quadrature
weights we make a rough error estimate for the approximated quadrature
using the Voronoi weights.

\begin{theorem}
  \label{theorem:voronoi_quadrature}
  Let $f\colon\mathcal M\to\C$ be Lipschitz continuous with the Lipschitz
  constant $L$.  Let further $\mathcal X$ be a set of nodes with mesh norm
  \begin{equation*}
    \delta_{\mathcal X} \coloneqq \max_{\bm y\in\mathcal M} \min_{\bm x\in\mathcal X} \dist(\bm y,\bm x)
  \end{equation*}
  and Voronoi weights $w_{\bm x}$. Then
  \begin{equation*}
    \left|\sum_{\bm x\in\mathcal X} w_{\bm x} f(\bm x)
      - \int_{\mathcal M}f(\bm y)\;\mathrm d\bm y\right|
    \le L\delta_{\mathcal X}\int_{\mathcal M}\mathrm d\bm y.
  \end{equation*}
\end{theorem}

\begin{proof}
  Since the disjoint union of all Voronoi cells $V_{\bm x}$ is $\mathcal M$ itself we can
  decompose the integral as follows
  \begin{equation*}
    \int_{\mathcal M} f(\bm y)\;\mathrm d\bm y
    = \sum_{\bm x\in\mathcal X} \int_{V_{\bm x}} f(\bm y)\;\mathrm d\bm y
    = \sum_{\bm x\in\mathcal X} \left(w_{\bm x}f(\bm x)+\int_{V_{\bm x}} f(\bm y)-f(\bm x)\;\mathrm d\bm y\right).
  \end{equation*}
  Noting that the maximal distance of $\bm x$ to any other node of the
  corresponding Voronoi cell $V_{\bm x}$ cannot exceed $\delta_{\mathcal X}$, we use
  the Lipschitz continuity to estimate the leftover integrand
  \begin{equation*}
    \left|\sum_{\bm x\in\mathcal X} w_{\bm x} f(\bm x) - \int_{\mathcal M}f(\bm y)\;\mathrm d\bm y\right|
    = \left|\sum_{\bm x\in\mathcal X}\int_{V_{\bm x}} f(\bm y)-f(\bm x)\;\mathrm d\bm y\right|
    \le L\delta_{\mathcal X}\int_{\mathcal M}\mathrm d\bm y.
  \end{equation*}
\end{proof}

\begin{remark}
  \begin{compactenum}[(i)]
  \item Theorem~\ref{theorem:voronoi_quadrature} states that the error of the
    quadrature formula gets small for smooth functions in the sense of a small
    Lipschitz constant and for small mesh norms, like for approximately
    equidistributed nodes.
  \item  Deterministic and probabilistic error estimates are available from \cite{Groechenig92} and \cite{BaGr03}, respectively.
  \item The Voronoi decomposition is dual to the Delaunay triangulation and
    thus can be computed in $\mathcal O(|\mathcal X|\log |\mathcal X|)$ for the Euclidean distance
    in $\dist(\cdot,\cdot)$.
  \end{compactenum}
\end{remark}

Given that the error of the Voronoi quadrature is small we obtain
approximately
\begin{align*}
  \F\herm\bm W\F
  &= \left[\sum_{\bm x\in\mathcal X} w_{\bm x} \e^{2\pi\mathrm i\bm n_1\bm x}
    \overline{\e^{2\pi\mathrm i\bm n_2\bm x}}\right]_{\bm n_1,\bm n_2\in \mathcal I}\\
  &\approx \left[\int_\T \e^{2\pi\mathrm i\bm n_1\bm x}\e^{-2\pi\mathrm i\bm n_2\bm
      x}\;\mathrm d\bm x\right]_{\bm n_1, \bm n_2\in \mathcal I}
  = \bm I \in \mathbb C^{\mathcal I \times \mathcal I}
\end{align*}
where $\bm I$ denotes the identity matrix. Inserting this into the definition
of the hat matrix $\bm H$ we have formally
\begin{equation}\label{eq:appr_hat}
  \bm H = \bm F \left(\F\herm \bm W \F + \lambda \bm{\hat  W}\right)^{-1} \F\herm \bm W
  \approx \bm F \left(\bm I + \lambda \bm{\hat  W}\right)^{-1} \F\herm \bm W
  =: \bm{\tilde H}.
\end{equation}
Analogously to Theorem~\ref{theorem:simplification_torus} we obtain for the
diagonal entries $\tilde h_{\bm x,\bm x}$ of the \emph{approximated hat matrix}
$\bm{\tilde H}$,
\begin{equation*}
  \tilde h_{\bm x,\bm x}
  = w_{\bm x} \sum_{\bm n \in \mathcal I} \frac{1}{1 + \lambda \hat w_{\bm n}}.
\end{equation*}
Together with Theorem~\ref{theorem:simp1} and
Definition~\ref{def:generalizedCV} this motivates the following definition of
approximated cross-validation scores.

\begin{definition}\label{def:cv_tilde}
  The approximated cross-validation score
  $\tilde P(\lambda)$ and the approximated generalized cross-validation score
  $\tilde V(\lambda)$ are defined by
  \begin{equation*}
    \tilde P(\lambda)
    = \sum_{\bm x \in \mathcal X}
    \frac{[\bm {H \bm f - \bm f}]_{\bm x}^{2}}{(1-\tilde h_{\bm x,\bm x})^{2}}
    \quad
    \text{and}
    \quad
    \tilde V(\lambda)
    = \sum_{\bm x \in \mathcal X}
    \frac{[\bm {H \bm f - \bm f}]_{\bm x}^{2}}{(1-\tilde h)^{2}},
  \end{equation*}
  where $\tilde h = \frac{1}{|\mathcal X|} \sum_{\bm x \in \mathcal X} \tilde h_{\bm x,\bm x}$.
\end{definition}

\begin{remark}
\label{remark:approximationisfast}
Algorithm~\ref{algo:cv} is easily modified to compute the approximated scores
by replacing all occurrences of $h_{\bm x, \bm x}$ by
$\tilde h_{\bm x, \bm x}$. The computationally most expensive part remains the
computation of the Tikhonov minimizer
$\bm{\tilde f} = \bm H \bm f = \F(\F\herm\bm W\F+\lambda\bm{\hat
  W})\inv\F\herm\bm W\f$.  Making use of the NFFT  \cite{KeKuPo09, nfft3}
  the
matrix-vector multiplications with $\F$ and $\F\herm$ can be performed with
$\mathcal O(|\mathcal I|\log |\mathcal I|+|\mathcal X|)$ numerical operations.
\end{remark}

In the remainder of this sections we illustrate that the approximated
cross-validation scores can be computed drastically faster while providing a
good approximation to the minimizer of the actual cross-validation score. To
this end we considered scattered sampling points on the one-dimensional torus
$\mathbb T$ as well as on the two-dimensional torus $\mathbb T^{2}$.  In order
generate sufficiently nonuniform sampling points we drew random samples with
respect to the uniform distribution on $[0,1]$ and $[0,1]^{2}$ and squared
them. This leads to sampling sets that are more dense towards the point $0$
and the edges $0 \times [0,1]$ and $[0,1] \times 0$.

In the one-dimensional example we used $|\mathcal X| = 128$ sampling
points and the index set $\mathcal I_{64}^{1d} = \{-32,\dots,31\}$. In the two-dimensional
example we chose $|\mathcal X| = 8192$ and
$\mathcal I_{64}^{2d} = \mathcal I_{64}^{1d}\times\mathcal I_{64}^{1d}$. In
both cases this corresponds to an oversampling factor of two.  As in
Subsection~\ref{sec:torus_equispaced} we chose as a test function the
\textsc{Matlab} peaks function with fixed second argument zero in the
one-dimensional case. Finally, we added $5\%$ Gaussian noise
as depicted in Figure~\ref{fig:t_nonequispaced_1d}, (a) and
Figure~\ref{fig:t_nonequispaced_2d}, (a).

\begin{figure}
  \centering

  \begin{subfigure}[t]{0.4\linewidth}

    \begin{tikzpicture}[font = \footnotesize]
      \begin{axis}[
        legend style={font=\tiny},
        legend pos = south east,
        scale only axis,
        width = 0.9\textwidth,
        height = 4.1cm,
        enlarge x limits = {abs = 0},
        xlabel = x ]

        \addplot[only marks, mark = *,mark options={scale = 0.4, color =
          black}] table[x index = 0, y index = 1]
        {data/t_nonequispaced_1d_noisy.dat};
        \addplot[mark=none,orange,thick] table[x index = 0, y index = 1]
        {data/t_nonequispaced_1d_reconstruction.dat};

        \legend{$\f = \F \hat \f +
          \bm\epsilon$,$\F \tilde \f_{\lambda}$}

      \end{axis}
    \end{tikzpicture}
    \caption{noisy input data $\f = \F \hat \f + \bm\epsilon$ and
      reconstruction $\F \tilde \f_{\lambda}$ with $\lambda$ set to the
      minimizer of $\tilde P(\lambda)$ }
  \end{subfigure}
  \quad
  \begin{subfigure}[t]{0.5\linewidth}
    \begin{tikzpicture}[font = \footnotesize]
      \begin{axis}[ scale only axis, width = 0.8\textwidth, height =
         4cm, xlabel = $\lambda$, xmode = log, axis y line* = left,
        enlarge x limits = {abs = 0}, ytick = {0.65655,0.65645}, yticklabels =
        {$10^{0.65655}$,$10^{0.65645}$}, legend pos = north east, legend cell
        align = left,
        xmin = 10^-4.5, xmax = 10^-1.6,
        ]
\addplot[no marks,thick] table[x index = 0, y index =
        1]{data/t_nonequispaced_1d.dat};

        \addplot [mark=*] coordinates {(5.1670650e-04,4.5336210e+00)};

      \end{axis}
      \begin{axis}[
        scale only axis,
        width = 0.8\textwidth,
        height = 4cm,
        ymax = 0.5,
        xmode = log,
        ymode = log,
        axis y line* = right,
        axis x line = none,
        xmin = 10^-4.5, xmax = 10^-1.6,
        axis line style = {orange},
        yticklabel style = {color = orange},
        enlarge x limits = {abs = 0},
        ytick = {0.4467,0.3548},yticklabels = {$10^{-0.35}$,$10^{-0.45}$},
        legend style={font=\tiny},
        ]
        \addplot[no marks,orange,thick] table[x index = 0, y index = 2]{data/t_nonequispaced_1d.dat};
        \addplot[no marks,orange,dotted,thick] table[x index = 0, y index = 1] {data/t_nonequispaced_1d_ocv_appr.dat};
        \addplot[no marks,orange,dashed,thick] table[x index = 0, y index = 4] {data/t_nonequispaced_1d.dat};
        \addplot[no marks,orange,dashdotted,thick] table[x index = 0, y index = 5] {data/t_nonequispaced_1d.dat};

        \legend{$P(\lambda)$, $\tilde P(\lambda)$, $V(\lambda)$,
          $\tilde V(\lambda)$}; \addlegendimage{solid,thick};
        \addlegendentry{$\lVert \tilde \f_{\lambda} - \hat \f \rVert_{2}$};

        \addplot [mark=*,color=orange] coordinates {(1.0481354e-03,3.3049875e-01)};
        \addplot [mark=*,color=orange] coordinates {(1.8456790e-03,3.3607991e-01)};
        \addplot [mark=*,color=orange] coordinates {(1.6022175e-03,3.5384227e-01)};
        \addplot [mark=*,color=orange] coordinates {(2.1261351e-03,3.5765573e-01)};

      \end{axis}
    \end{tikzpicture}

    \caption{approximation error
      $\lVert \tilde \f_{\lambda} - \hat \f \rVert_{2}$ (black) and cross-validation
      scores (orange)}

  \end{subfigure}

  \caption{Approximation from nonequispaced data: Comparison of the ordinary
    cross-validation score $P(\lambda)$ and the generalized cross-validation
    score $V(\lambda)$ with their approximations $\tilde P(\lambda)$ and
    $\tilde V(\lambda)$ and the approximation error.}
  \label{fig:t_nonequispaced_1d}

\end{figure}
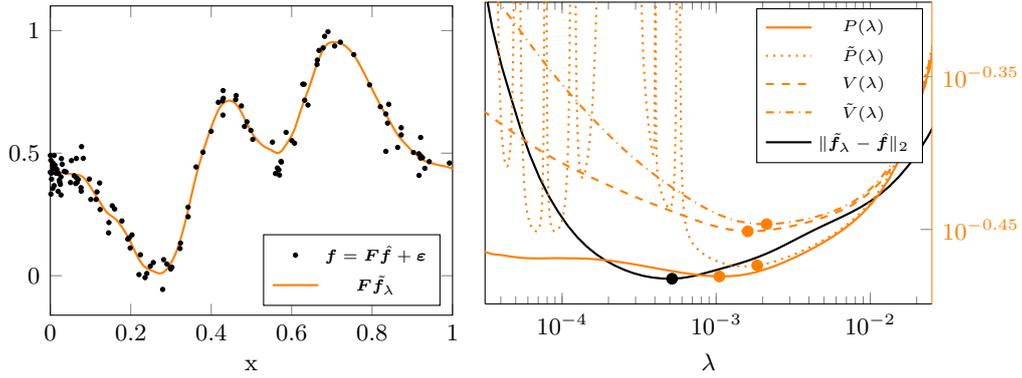

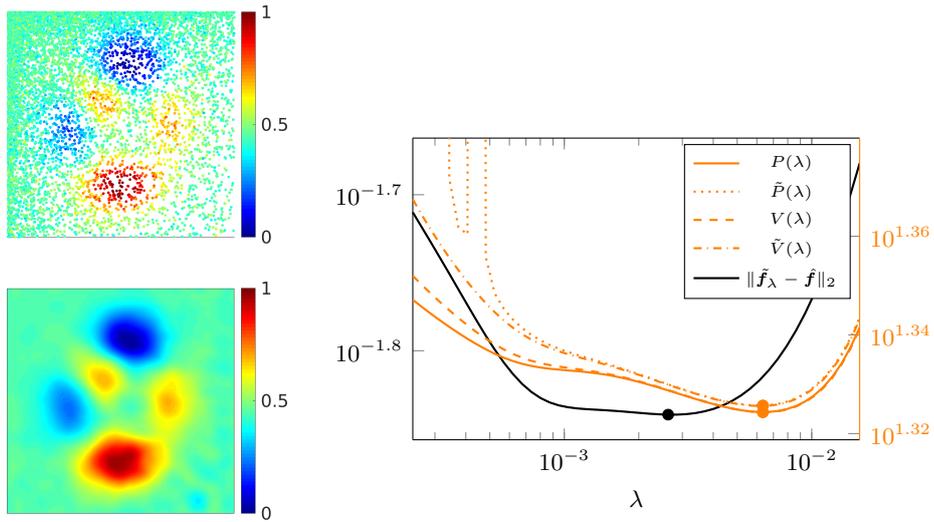
\begin{figure}
  \centering

  \begin{subfigure}[b]{0.25\textwidth}
      \includegraphics[width=\textwidth]{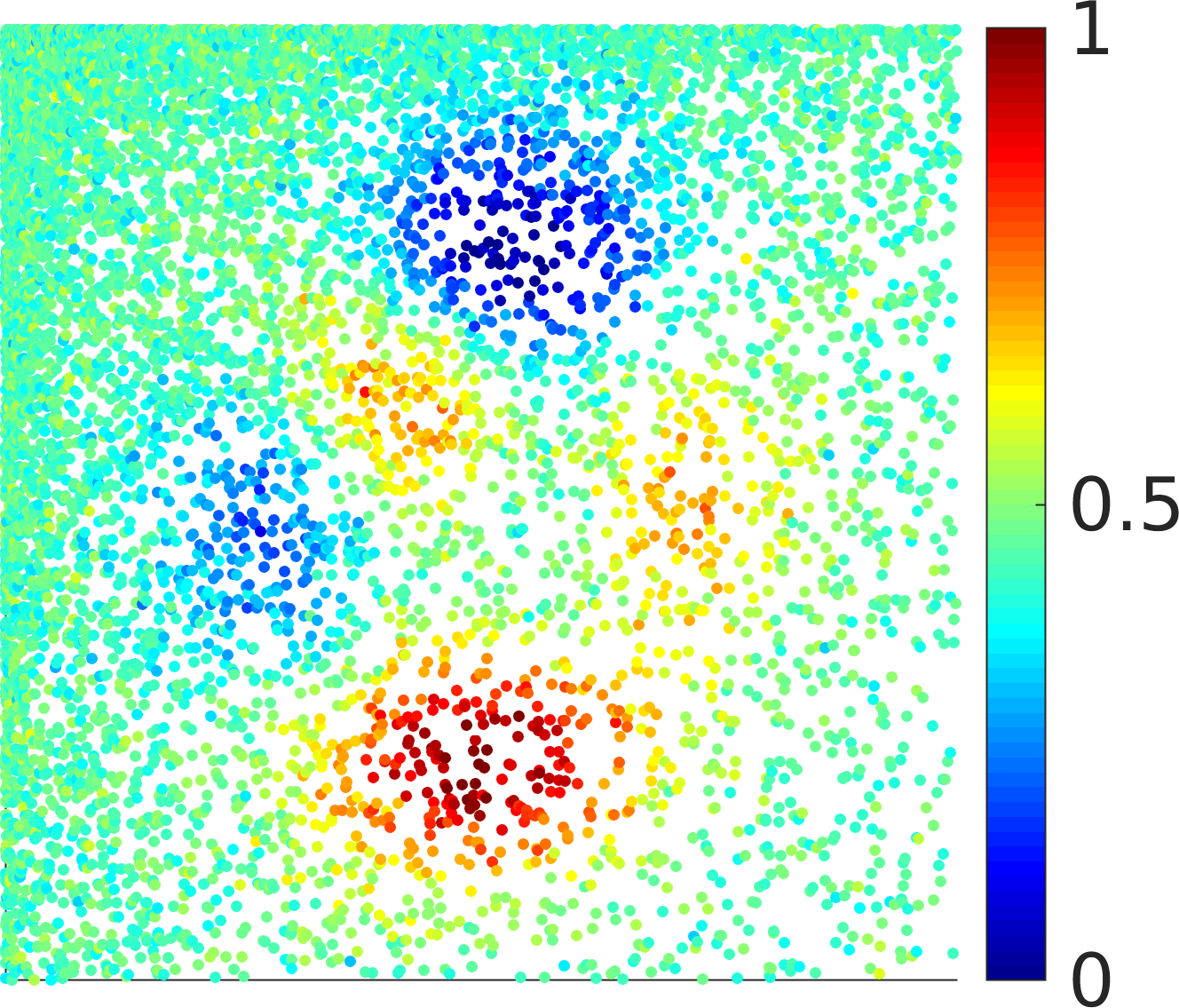}
      \\ \\
      \includegraphics[width=\textwidth]{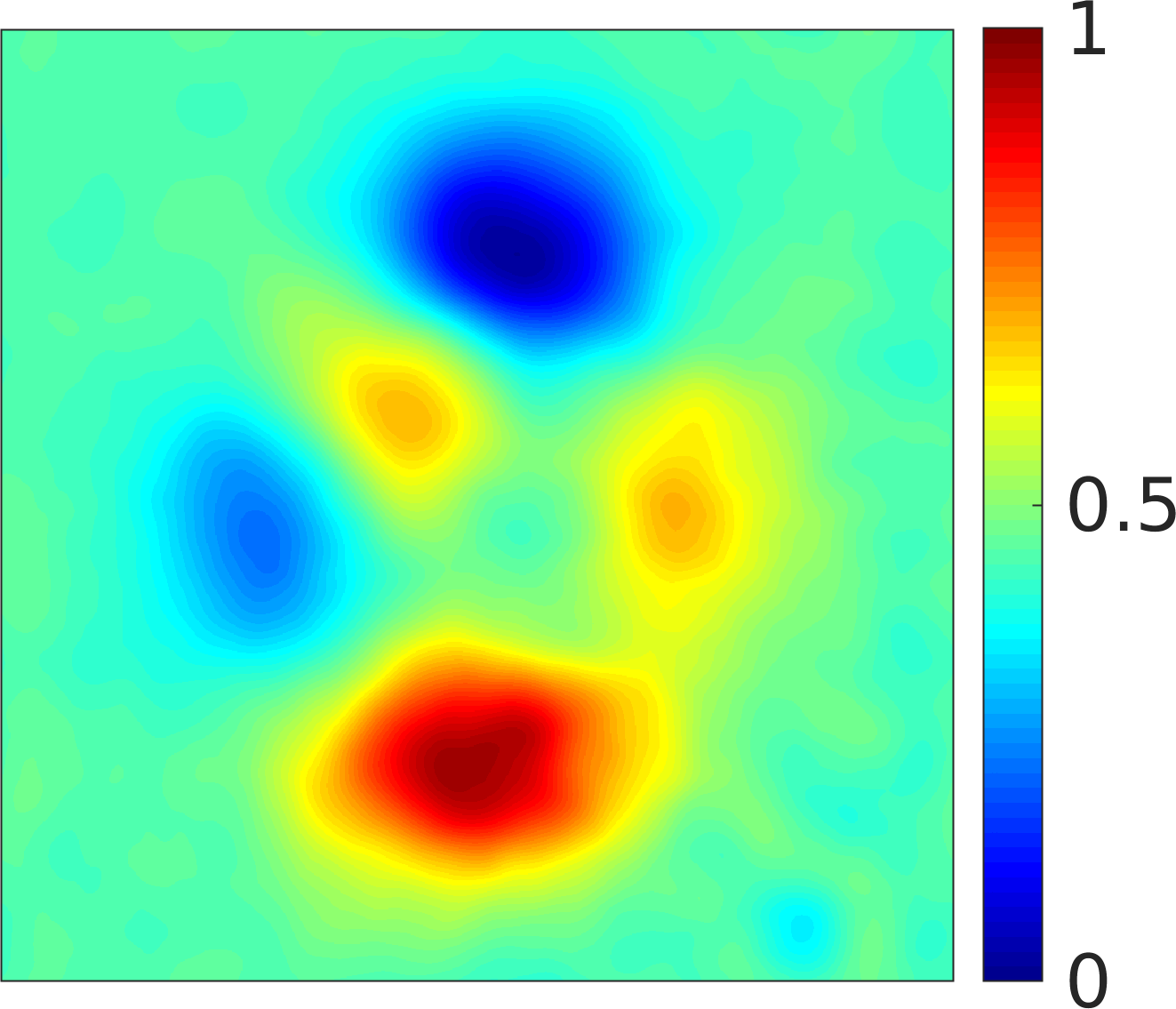}
    \caption{noisy input data and reconstruction}
  \end{subfigure}
  \quad
  \begin{subfigure}[b]{0.5\textwidth}
    \begin{tikzpicture}[font = \footnotesize]
      \begin{axis}[
        scale only axis,
        width             = 0.8\textwidth, height = 4cm,
        xlabel            = $\lambda$,
        xmode             = log,
        ymode             = log,
        axis y line*      = left,
        enlarge x limits  = {abs = 0},
        ytick             = {1.58489,1.9952},
        yticklabels       = {$10^{-1.8}$,$10^{-1.7}$},
        ]
\addplot[no marks,thick] table[x index = 0, y index = 1] {data/t_nonequispaced_2d.dat};
        \addplot [mark=*] coordinates {( 2.6287113e-03,1.4423143)};

      \end{axis}
      \begin{axis}[
        scale only axis,
        width             = 0.8\textwidth, height = 4cm,
        ymax              = 24,
        xmode             = log,
        ymode             = log,
        axis y line*      = right,
        axis x line       = none,
        axis line style   = {orange},
        yticklabel style  = {color = orange},
        enlarge x limits  = {abs = 0},
        legend style={font=\tiny},
        ]
        \addplot[no marks,orange,thick] table[x index = 0, y index = 2] {data/t_nonequispaced_2d.dat};
        \addplot[no marks,orange,dotted,thick] table[x index = 0, y index = 3] {data/t_nonequispaced_2d.dat};
        \addplot[no marks,orange,dashed,thick] table[x index = 0, y index = 4] {data/t_nonequispaced_2d.dat};
        \addplot[no marks,orange,dashdotted,thick] table[x index = 0, y index
        = 5] {data/t_nonequispaced_2d.dat};

        \legend{$P(\lambda)$, $\tilde P(\lambda)$, $V(\lambda)$,
          $\tilde V(\lambda)$}; \addlegendimage{solid,thick};
        \addlegendentry{$\lVert \tilde \f_{\lambda} - \hat \f \rVert_{2}$};

        \addplot [mark=*,color=orange] coordinates {(6.3457218e-03,2.1183233e+01)};
        \addplot [mark=*,color=orange] coordinates {(6.3457218e-03,2.1117441e+01)};
        \addplot [mark=*,color=orange] coordinates {(6.3457218e-03,2.1177888e+01)};
        \addplot [mark=*,color=orange] coordinates {(6.3457218e-03,2.1115331e+01
          )};
      \end{axis}
    \end{tikzpicture}
    \caption{approximation error
      $\lVert \tilde \f_{\lambda} - \hat \f \rVert_{2}$ (black) and cross-validation
      scores (orange)}
  \end{subfigure}

  \caption{Approximation from two-dimensional nonequispaced data: Comparison
    of the ordinary cross-validation score $P(\lambda)$ and the generalized
    cross-validation score $V(\lambda)$ with their approximations
    $\tilde P(\lambda)$ and $\tilde V(\lambda)$ and the approximation error.}
  \label{fig:t_nonequispaced_2d}
\end{figure}

As in both cases the weights $w_{\bm x}$ are far from uniform we may expect
some difference between the ordinary cross-validation score $P(\lambda)$ and
the generalized cross-validation score $V(\lambda)$. This can be observed in
the one-dimensional example, cf.~Figure~\ref{fig:t_nonequispaced_1d}. In the
two-dimensional example both functionals coincide surprisingly well,
cf.~Figure~\ref{fig:t_nonequispaced_2d}. Furthermore, the weights are
 completely uncorrelated to the homogeneous noise level and, thus, contradict
 with Remark~\ref{remark:problemalteringweights}.

Judging the approximation of the exact cross-validation scores $P(\lambda)$
and $V(\lambda)$ by the approximated scores $\tilde P(\lambda)$ and
$\tilde V(\lambda)$ we observe that for small regularization parameters
$\lambda$ the score $\tilde P(\lambda)$ contains several peaks for both examples.
These peaks occur because we overestimate the diagonal entries $h_{x,x}$ such that they attain values around one and summands of the ordinary cross-validation score \eqref{eq:cv_simple} diverge.
In contrast Lemma~\ref{lemma:h<1} says that these diagonal entries are always smaller than one.
Nevertheless, the minimizer of all four functionals $P,\tilde P, V,
\tilde V$ are very close together for the two-dimensional example while for
the one-dimensional example the minimizer of $P$ and $\tilde P$ are
closer to the $L_2(\T)$-optimal regularization parameter compared to the
minimizer of $V$ and $\tilde V$. A natural idea to avoid the oscillatory
regions of the functional $\tilde P$ would be to use the minimizer of $\tilde
V$ as initial guess for minimizing $\tilde P$.

The central reason for preferring the functional $\tilde P$ and $\tilde V$
over the functionals $P$, $V$ is that they are faster to
compute. Indeed, if we fix the number of iterations for computing the Tikhonov minimizer, $P(\lambda)$ and $V(\lambda)$ can be acquired in
$\mathcal O(|\mathcal I| |\mathcal X| \log |\mathcal I| + |\mathcal X|^{2})$
numerical operations for a single regularization parameter $\lambda$, which compares to
$\mathcal O(|\mathcal I| \log |\mathcal I| + |\mathcal X|)$ numerical
operations for the evaluation of $\tilde P$ and $\tilde V$. In our toy example
the computation of $P$ and $V$ took approximately $1.61$ for the one-dimensional and $1278$ seconds for the two-dimensional case,
while the computation of $\tilde P$ and $\tilde V$ was performed within
$0.02$ and $0.16$ seconds averaged over all tested $\lambda$.

\section{Cross-validation on the unit interval}
\label{sec:i}

In this section, we consider the cross-validation scores for nonperiodic
functions on the unit interval $[-1,1]$ with respect to the \emph{Chebyshev polynomials}
$$
  T_n(x) = \cos(n\arccos x)
  \quad n = 0,\dots,N-1
$$
up to polynomial degree $N\in\mathbb N$.
In this setting the Fourier matrix becomes
$$
  \bm F = \left[ T_n(x) \right]_{x\in\mathcal X,n=0,\dots,N-1}
$$
for a set of nodes $\mathcal X$.

\subsection{Exact Quadrature}
\label{sec:exact-quadrature}

Similarly as for functions on the torus we consider the case of exact quadrature first.
Therefore we remind that the Chebyshev polynomials are orthogonal with respect to the inner product
$$
 (f,g) = \int_{-1}^1 \frac{f(x)g(x)}{\sqrt{1-x^2}}\;\mathrm dx
$$
and are normalized such that
$$
  (T_{n_1},T_{n_2})
  = \begin{cases}
    \pi &\colon n_1=n_2=0,\\
    \pi/2 &\colon n_1=n_2\neq 0,\\
    0 &\colon n_1\neq n_2.
  \end{cases}
$$

Assuming that the nodes $\mathcal X$ and weights $\bm W$ form a quadrature rule that is exact up to polynomial degree $2N-2$ the diagonal entries of the hat matrix $\bm H$ can be given expicitly using the following theorem.

\begin{theorem}
\label{theorem:i_diagonals_quadrature}
Let the nodes $\mathcal X$ and the weights $\bm W = \diag(w_x)_{x\in\mathcal X}$ form a quadrature rule which is exact up to polynomial degree $2N-2$ and let $\bm{\hat W}=\diag(\hat w_0,\dots,\hat w_{N-1})$ be the weights in frequency domain.
Then the diagonal entries $h_{x,x}$ of the hat matrix $\bm H = \F(\F\herm\bm W\F+\lambda\bm{\hat W})\inv\F\bm W$ corresponding to the nodes $x\in\mathcal X$ satisfy
$$
  h_{x,x} = \frac{w_x}2 \left(\sum_{n=1}^{N-1} \frac 1{\pi/2+\lambda \hat w_n} \cos(2n\arccos x)
  + \sum_{n=1}^{N-1}\frac 1{\pi/2+\lambda\hat w_n}
  + \frac 2{\pi+\lambda\hat w_0}\right).
$$
\end{theorem}

\begin{proof}
Similar to Theorem~\ref{theorem:simplification_torus} we obtain $\F\herm\bm W\F+\lambda\bm{\hat W} = \diag(\pi+\lambda\hat w_0, \pi/2+\lambda\hat w_1, \dots, \pi/2+\lambda\hat w_{N-1})$.
Putting this into the formula for the diagonal elements of the hat matrix obtain
$$
  h_{x,x} = w_x\left(\sum_{n=1}^{N-1} \frac 1{\pi/2+\lambda\hat w_n} \cos^2(n\arccos x) + \frac 1{\pi+\lambda\hat w_0}\right).
$$
In combination with the addition theorem $\cos(2x) = 2\cos^2 x-1$ this proves the assertion.
\end{proof}

\subsection{Chebyshev nodes}

The most basic examples of an exact quadrature formula on the interval is
probably quadrature at Chebyshev nodes. In order to restate
Theorem~\ref{theorem:i_diagonals_quadrature} for this case we require the
discrete cosine transforms from first up to third kind
\begin{align*}
  \bm C_{N+1}^\mathrm I
  &\coloneqq \sqrt\frac 2N \left[\bm\gamma_N(n)\bm\gamma_N(m)\cos\left(\frac{nm\pi}N\right)\right]_{n,m=0}^N, \\
  \bm C_N^\mathrm{II}
  &\coloneqq \sqrt\frac 2N
    \left[\bm\gamma_N(n)\cos\left(\frac{n(2m+1)\pi}{2N}\right)\right]_{n,m=0}^{N-1},
    \quad
    \bm C_N^\mathrm{III}
    \coloneqq\left(\bm C_N^\mathrm{II}\right)\transp
\end{align*}
with $\bm\gamma(0)=\bm\gamma(N)\coloneqq \sqrt 2/2$ and
$\bm\gamma(n)\coloneqq 1$ for $n=1,\dots,N-1$.  The corresponding
matrix-vector products can be calculated using $\mathcal O(N \log N)$
arithmetic operations, cf.~\cite[Chapter 6]{PlPoStTa18}.

Using the fact that $\bm C_N^\mathrm{III}$ is orthonormal, cf.~\cite[Sec. 3.5]{PlPoStTa18}, we acquire
$$
  \bm I
  = C_N^\mathrm{II} C_N^\mathrm{III}
  = \left[\gamma_N(n)\cos\left(\frac{n(2m+1)\pi}{2N}\right)\right]_{n,m=0}^{N-1}
    \frac 2N \left[\gamma_N(n)\cos\left(\frac{n(2m+1)\pi}{2N}\right)\right]_{m,n=0}^{N-1}.
$$
If we multiply with $\diag(\sqrt \pi,\sqrt{\pi/2},\dots,\sqrt{\pi/2})$ from both sides we obtain
$$
  \diag(\pi,\pi/2,\dots,\pi/2)
  = \left[\cos\left(\frac{n(2m+1)\pi}{2N}\right)\right]_{n,m=0}^{N-1}
    \frac \pi N \left[\cos\left(\frac{n(2m+1)\pi}{2N}\right)\right]_{m,n=0}^{N-1}.
$$
Putting this into the form $\F\herm\bm W\F$ we see that the \emph{Chebyshev nodes of first kind}
$$
  x_m = \cos\left(\frac{(2m+1)\pi}{2N}\right),\quad m=0,\dots,N-1
$$
and the uniform weights $w_x = \pi/N$ form a quadrature rule which is exact up to degree $2N-2$.

For these specific nodes Theorem~\ref{theorem:i_diagonals_quadrature} simplifies to:

\begin{theorem}
  \label{theorem:i_diagonals}
  Let $\mathcal X=\{x_1,\dots,x_m\}$ be the Chebyshev nodes of first kind and $w_{x_m} = \pi/N$.
  Then the diagonal entries $h_{x_m,x_m}$ of the hat matrix \eqref{eq:hat} can be written as
  \begin{equation*}
    h_{x_m,x_m}
    = \frac{w_{x_m}}2 \left(\frac {\sqrt{N/2}}{\gamma_{2N}(m)} \left[\bm C_{2N+1}^\mathrm I \bm b\right]_{2m+1}
    + \sum_{n=1}^{N-1}\frac 1{\pi/2+\lambda\hat w_n}\right), \quad m=0,\dots,N-1
  \end{equation*}
  with
  \begin{equation*}
    \bm b = (b_0,\dots,b_{2N})\transp
    = \left(\frac{2\sqrt 2}{\pi+\lambda\hat w_0},0,\frac{1}{\pi/2+\lambda\hat w_1},0,\dots,\frac{1}{\pi/2+\lambda\hat w_{N-1}},0,0\right)\transp.
  \end{equation*}
\end{theorem}

\begin{proof}
According to Theorem~\ref{theorem:i_diagonals_quadrature} we have
$$
  h_{x_m,x_m}
  = \frac{w_{x_m}}2 \left(\sum_{n=1}^{N-1} \frac 1{\pi/2+\lambda \hat w_n} \cos\left(\frac{n(2m+1)\pi}{N}\right)
  + \sum_{n=1}^{N-1}\frac 1{\pi/2+\lambda\hat w_n}
  + \frac 2{\pi+\lambda\hat w_0}\right).
$$
Using the coefficients $\bm b$ the first sum can be expressed with twice the frequency
$$
  h_{x_m,x_m}
  = \frac{w_{x_m}}2 \left(\frac 1{\gamma_{2N}(m)}\sum_{n=0}^{2N} b_n \gamma_{2N}(n)\gamma_{2N}(m) \cos\left(\frac{n(2m+1)\pi}{2N}\right)
  + \sum_{n=1}^{N-1}\frac 1{\pi/2+\lambda\hat w_n}\right),
$$
which is the cosine transform of first kind.
\end{proof}

\begin{corollary}
\label{corollary:i_fast_exact}
For fixed $\lambda$ the ordinary cross-validation score $P(\lambda)$ and the generalized cross-validation score $V(\lambda)$ on the unit intervall for Chebyshev nodes of first kind can be computed in $\mathcal O(N\log N)$ using Algorithm~\ref{algo:cv}.
\end{corollary}
\begin{proof}
Because multiplying with $\F$ and $\F\herm$ can be done using the fast discrete cosine transform, we see that applying the hat matrix can be done in $\mathcal O(N\log N)$ and Theorem~\ref{theorem:i_diagonals} allows us to compute the diagonal entries of the hat matrix in $\mathcal O(N\log N)$.
\end{proof}

To emphasize our results numerically we chose the peaks sample function $f$
from \mbox{\textsc{Matlab}} and fixed the second argument to zero.  We
evaluated $f$ in $N=128$ Chebyshev nodes and added 5\% Gaussian noise as one
can see in Figure~\ref{fig:unitintervall}, (a).  To choose $\bm{\hat W}$ we
made use of the following statement from \cite[Theorem 7.1]{Tref13} which
relates the smoothness of $f$ to the decay of the Chebyshev coefficients
$\bm a$: If for $\nu\ge 0$ the derivatives up to $f^{(\nu-1)}$ are absolute
continuous and $f^{(\nu)}$ has bounded variation $V$ then
$|a_k| \le 2V/(\pi(k-\nu)^{\nu+1})$.  Because in general we do not know
anything about the smoothness of the function $f$ we chose $\hat w_n=n^3$ as
weights which corresponds to a function with one absolute continuous
derivative.  We used Algorithm~\ref{algo:cv} to calculate the ordinary
cross-validation score $P(\lambda)$ and the generalized cross-validation score
$V(\lambda)$ for $\lambda\in[2^{-16},2^{-11}]$ and plotted the regularization
for the $\lambda$ with the smallest corresponding ordinary cross-validation
score as one can see in Figure~\ref{fig:unitintervall}, (b).

\begin{figure}
  \centering

  \begin{subfigure}[t]{0.4\linewidth}
    \begin{tikzpicture}[font = \footnotesize]
      \begin{axis}[
        scale only axis,
        width             = 0.9\textwidth, height = 4.1cm,
        enlarge x limits  = {abs = 0},
        xlabel            = $x$,
        legend style={font=\tiny},
        legend pos = south east,
        ]
        \addplot[mark = *,mark options={scale = 0.4, color = black},only marks]
        table[x index = 0, y index = 1] {data/i_example.dat};
        \addplot[mark = none,orange,thick] table[x index = 0, y index = 2]
        {data/i_example.dat};

        \legend{$\f = \F \hat \f + \bm\epsilon$,$\F \tilde \f_{\lambda}$}
      \end{axis}
    \end{tikzpicture}
    \caption{noisy input data $\f = \F \hat \f + \bm\epsilon$ and
      reconstruction $\F \tilde \f_{\lambda}$ with $\lambda$ set to the
      minimizer of $P(\lambda)$ }
  \end{subfigure}
  \quad
  \begin{subfigure}[t]{0.5\linewidth}
    \begin{tikzpicture}[font = \footnotesize]
      \begin{axis}[
        scale only axis,
        width             = 0.7\textwidth, height = 4cm,
        xlabel            = $\lambda$,
        xmin = 3.82e-06, xmax = 9.76e-04,
        xmode             = log,
        ymode             = log,
        axis y line*      = left,
        enlarge x limits  = {abs = 0},
        legend pos        = north east,
        ]
        \addplot[no marks,thick] table[x index = 0, y index = 1] {data/i.dat};
        \addplot [mark=*] coordinates {(2.0027719e-04,5.2918152e-02)};
      \end{axis}
      \begin{axis}[
        scale only axis,
        width             = 0.7\textwidth, height = 4cm,
        xmode             = log,
        xmin = 3.82e-06, xmax = 9.76e-04,
        ymode             = log,
        axis y line*      = right,
        axis x line       = none,
        axis line style   = {orange},
        yticklabel style  = {color = orange},
        enlarge x limits  = {abs = 0},
        legend style={font=\tiny},
        ]
        \addplot[no marks,orange,thick] table[x index = 0, y index = 2] {data/i.dat};
        \addplot[no marks,orange,dashed,thick] table[x index = 0, y index = 3]{data/i.dat};

        \legend{$P(\lambda)$, $V(\lambda)$};
        \addlegendimage{solid,thick};
        \addlegendentry{$\|\F\ftilde_\lambda-f\|_{L_2}$};
        \addplot [mark=*,color=orange] coordinates {( 2.5114642e-04,4.9769009e-01)};
        \addplot [mark=*,color=orange] coordinates {( 2.5114642e-04,4.9801400e-01)};
      \end{axis}
    \end{tikzpicture}
    \caption{ $L_2([-1,1],(1-x^2)^{-1/2})$ approximation error (black) and cross-validation
      scores (orange)}
  \end{subfigure}
  \caption{Approximation on the unit interval from data at Chebyshev nodes:
    Comparison of the ordinary cross-validation score $P(\lambda)$ and the
    generalized cross-validation score $V(\lambda)$ with the
    approximation error.}

  \label{fig:unitintervall}
\end{figure}
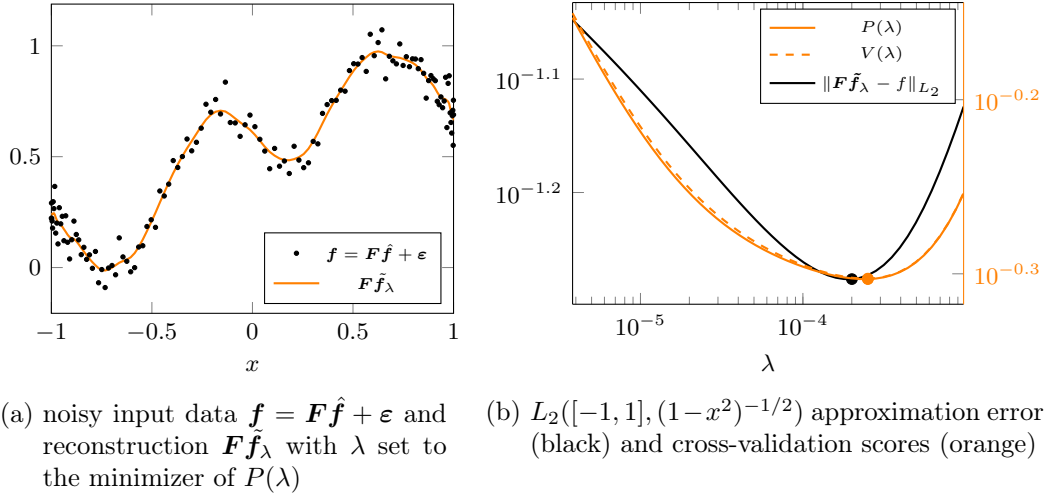

We observe that the ordinary cross-validation score and the generalized
cross-validation score differ only slightly and their minima are close to the
$L_{2}([-1,1],(1-x^{2})^{-1/2})$ optimal $\lambda$.

\subsection{Approximative Quadrature}
\label{sec:appr-quadr}

In this section, we consider arbitrary, ordered nodes
$x_{m} \in \mathcal X \subset [-1,1]$, $m=0,\ldots,M$.
The corresponding cosine transforms can be computed using the nonequispaced discrete cosine transform, cf.~\cite{fepo02}, in $\mathcal O(N\log N+|\mathcal X|)$ where $N$ is the bandwidth.
As in Section~\ref{subs:torus_app} we determine approximate quadrature
weights $w_{x_m}$ for $m = 0,\ldots,M$ that allow us to efficiently estimate the
diagonal entries of the hat matrix $\bm H$.  Since we consider the unit
interval with the non-uniform weight function $(1-x^{2})^{-\frac12}$ it is not
a good idea to compute Voronoi weights directly. Instead, we consider the
corresponding periodic approximation problem on the unit circle with constant
weight by substituting $y_{m} = \arccos x_{m} \in [0,\pi]$ and use Voronoi
weights with respect to $y_{m}$, i.e.,
\begin{equation}
  \label{eq:wx_interval}
  w_{x_m} =
  \begin{cases}
    \frac{y_{0} + y_{1}}2, \quad & m = 0,\\
    \frac{y_{m+1} - y_{m-1}}2, \quad &m = 1,\ldots,M-1,\\
    \pi-\frac{y_{M-1} + y_{M}}2, \quad & m = M.
  \end{cases}
\end{equation}

\begin{remark}
  Let $x_{m}$ be the Chebyshev nodes of first kind. Then the quadrature
  weights \eqref{eq:wx_interval} coincide with the exact quadrature weights
  from Section~\ref{sec:exact-quadrature}.
\end{remark}

Analogously to Section~\ref{subs:torus_app} we use the approxmiated hat matrix $\bm{\tilde H}$ from \eqref{eq:appr_hat} for ease of computation.

\begin{remark}
  \label{remark:i_fast_appr}
  Using Theorem~\ref{theorem:i_diagonals_quadrature} the diagonal entries can be calculated efficiently with one nonequispaced discrete cosine transform.

  With the given tools we can modify Algorithm~\ref{algo:cv} to compute
  $\tilde P(\lambda)$ and $\tilde V(\lambda)$ from
  Definition~\ref{def:cv_tilde} in $\mathcal O(N\log N+|\mathcal X|)$ arithmetic
  operations given a fixed number of iterations to compute the Tikhonov
  minimizer.
\end{remark}

To exemplify our results we chose $128$ uniformly distributed nodes on the
unit interval which we perturbed by $5$\% Gaussian noise. Note that
uniformly distributed nodes are far from optimal in the setting of polynomial
interpolation on the interval. As in case of exact quadrature we set the
bandwith equal to the number of nodes, i.e., $N = |\mathcal X|$.  As the Voronoi weights
resamble quadrature weights the choice of the bandwidth $N$ is critical
because in the case of $|\mathcal X|<N$ one can not expect to get an exact quadrature
formula.  As test function we used again the \textsc{Matlab} peaks function
with fixed second argument.  Then we computed $P(\lambda)$,
$\tilde P(\lambda)$, $V(\lambda)$, and $\tilde V(\lambda)$ for
$\lambda\in[2^{-18},2^{-11}]$.  The results can be seen in
Figure~\ref{fig:i_nonequispaced}.  \begin{figure}
  \centering

  \begin{subfigure}[t]{0.4\linewidth}

    \begin{tikzpicture}[font = \footnotesize]
      \begin{axis}[
        legend style={font=\tiny},
        legend pos = south east,
        scale only axis,
        width = 0.9\textwidth,
        height = 4.1cm,
        enlarge x limits = {abs = 0},
        xlabel = x ]

        \addplot[only marks, mark = *,mark options={scale = 0.4, color =
          black}] table[x index = 0, y index = 1]
        {data/i_nonequispaced_noisy.dat};
        \addplot[mark=none,orange,thick] table[x index = 0, y index = 1]
        {data/i_nonequispaced_reconstruction.dat};

        \legend{$\f = \F \hat \f +
          \bm\epsilon$,$\F \tilde \f_{\lambda}$}

      \end{axis}
    \end{tikzpicture}
    \caption{noisy input data $\f = \F \hat \f + \bm\epsilon$ and
      reconstruction $\F \tilde \f_{\lambda}$ with $\lambda$ set to the
      minimizer of $\tilde P(\lambda)$ }
  \end{subfigure}
  \quad
  \begin{subfigure}[t]{0.5\linewidth}
    \begin{tikzpicture}[font = \footnotesize]
      \begin{axis}[ scale only axis, width = 0.8\textwidth, height =
         4cm, xlabel = $\lambda$, xmode = log, ymode = log, axis y line* = left,
        enlarge x limits = {abs = 0}, legend pos = north east, legend cell
        align = left,
        xmin = 1.5259e-5, xmax = 9.7656e-4,
        ]
\addplot[no marks,thick] table[x index = 0, y index =
        1]{data/i_nonequispaced.dat};

        \addplot [mark=*] coordinates {( 2.3070987e-04,5.4348358e-02)};

      \end{axis}
      \begin{axis}[
        scale only axis,
        width = 0.8\textwidth,
        height = 4cm,
        ymax = 0.5,
        xmode = log,
        ymode = log,
        axis y line* = right,
        axis x line = none,
        axis line style = {orange},
        yticklabel style = {color = orange},
        enlarge x limits = {abs = 0},
        ytick = {0.4467,0.3548},yticklabels = {$10^{-0.35}$,$10^{-0.45}$},
        legend style={font=\tiny},
        xmin = 1.5259e-5, xmax = 9.7656e-4,
        ]
        \addplot[no marks,orange,thick] table[x index = 0, y index = 2]{data/i_nonequispaced.dat};
        \addplot[no marks,orange,dotted,thick] table[x index = 0, y index = 1] {data/i_nonequispaced_ocv_appr.dat};
        \addplot[no marks,orange,dashed,thick] table[x index = 0, y index = 4] {data/i_nonequispaced.dat};
        \addplot[no marks,orange,dashdotted,thick] table[x index = 0, y index = 5] {data/i_nonequispaced.dat};

        \legend{$P(\lambda)$, $\tilde P(\lambda)$, $V(\lambda)$,
          $\tilde V(\lambda)$}; \addlegendimage{solid,thick};
        \addlegendentry{$\|\F\ftilde_\lambda-f\|_{L_2}$};

        \addplot [mark=*,color=orange] coordinates {( 9.8732019e-05,3.0217060e-01)};
        \addplot [mark=*,color=orange] coordinates {( 2.7339326e-04,3.3541911e-01)};
        \addplot [mark=*,color=orange] coordinates {( 1.3864406e-04,3.1443469e-01)};
        \addplot [mark=*,color=orange] coordinates {( 1.9469040e-04,3.3190453e-01)};

      \end{axis}
    \end{tikzpicture}

    \caption{$L_2([-1,1],(1-x^2)^{-1/2})$ approximation error (black) and cross-validation
      scores (orange)}

  \end{subfigure}

  \caption{Approximation from nonequispaced data: Comparison of the ordinary
    cross-validation score $P(\lambda)$ and the generalized cross-validation
    score $V(\lambda)$ with their approximations $\tilde P(\lambda)$ and
    $\tilde V(\lambda)$ and the approximation error.}
  \label{fig:i_nonequispaced}

\end{figure}
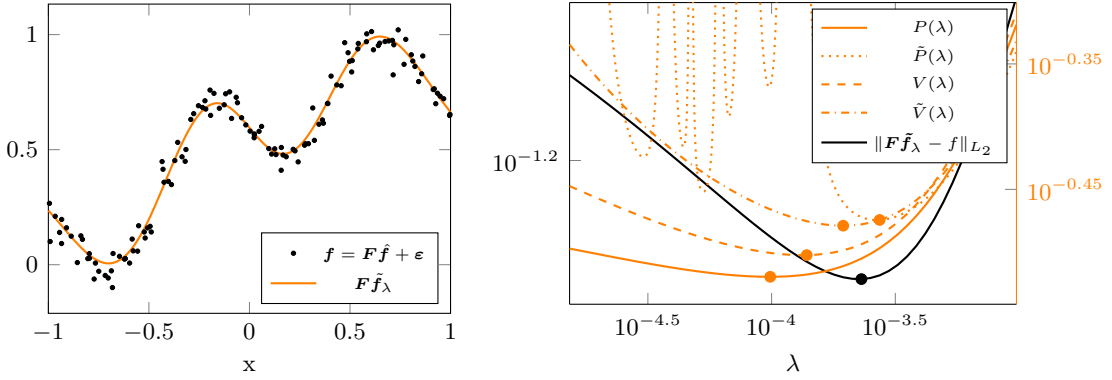

We observe that all cross-validation scores follow the shape of the
$L_{2}([-1,1],(1-x^{2})^{-1/2})$-error and their minima are close to the
optimal $\lambda$. Again, $\tilde P(\lambda)$ is affected by oscillations for
small $\lambda$ which are caused by diagonal entries $h_{x_m,x_m}$ close to $1$.
The computation of the exact $P(\lambda)$ and $V(\lambda)$ averaged over all
$\lambda$ takes $4.07$ seconds whereas the approximated $\tilde P(\lambda)$
and $\tilde V(\lambda)$ outperform this with $0.04$ seconds.

\section{Cross-validation on the two-dimensional sphere}
\label{sec:s2}

Approximation on the two-dimensional sphere
$\S \coloneqq \{\bm x \in \mathbb R^3:\|\bm x\|_2 = 1\}$ has been subject of
mathematical research for a long time. The base for approximation from
scattered data is formed by positive quadrature rules,
Marcinkiewicz-Zygmund inequalities which are investigated in the papers
\cite{ChXu92,MhNaWa01,GiMh07,BoKuPo07}, and by bounds for best approximations
\cite{slwo99,wosl01,HeSl06}. Based on these result the relationship between
the mesh norm, the separation distance of the sampling points, and optimal
approximation rates has been analyzed in the papers
\cite{FiTh06,Ku07,KeKuPo06b,ThBa18}. Approximation from noisy data has been
considered in \cite{AnChSlWo12} and a priori and a posteriori estimates of the
approximation error with respect to the regularization parameter have been
proven in \cite{PeSlTk14}.

Following the approach of the previous sections we again consider the weighted
Tikhonov functional \eqref{eq:tikhonov}. The analogue of the exponential
functions become the spherical harmonics
$\{Y_{n,k}\}_{n=0,\dots,\infty,k=-n,\dots,n}$,
cf.~\cite{frgesc,AtHa12,Mic13,DaXu2013}, which we assume to be normalized such
that they form an orthonormal basis in $L_2(\S)$. For nodes
$\mathcal X\subset\S$ and a maximum polynomial
degree $N \in \mathbb N$ the Fourier matrix $\F$ becomes
\begin{equation*}
  \F = [Y_{n,k}(\bm x)]_{\bm x\in\mathcal X;n=0,\dots,N,k=-n,\dots,n}.
\end{equation*}
As for the weight matrix $\bm{\hat W}$ in Fourier space we consider isotropic
weights
$\bm{\hat W} = \diag(\hat w_{n,k})_{h=0,\ldots,N,\, k=-n,\ldots,n}$
that depend only the polynomial degree, i.e.,
$\hat w_{n,k} = \hat w_{n}$.

\subsection{Exact Quadrature}

There are several approaches to exact quadrature on the two-dimensional sphere. The most
direct approach is to consider tensor products of Gauss quadrature rules on the
circle and the unit interval $[-1,1]$, cf.~\cite[Section 9.6]{PlPoStTa18}. A
relaxation of this idea is to choose the points equally spaced at fixed
latitudinal circles which also allows for an explicit computation of the
quadrature weights, cf.~\cite{ro09}.

A second approach is to choose the quadrature nodes approximately uniform
and determine the quadrature weights by solving a linear system of
equations. Given that the quadrature nodes are sufficiently well separated and
the oversampling factor is sufficiently high, the resulting quadrature
weights can guarantied to be nonnegative, cf.~\cite{MhNaWa01}. The computation
of these quadrature weights can be implemented efficiently using fast
spherical Fourier techniques, cf.~\cite{kupo02,KePo06,nfft3}.

A third approach, called Chebyshev quadrature, consists of fixing the weights
to be constant and seeking quadrature nodes with a high degree of
exactness. The resulting nodes are known as spherical t-designs. Efficient
algorithms for their computation are described in \cite{GrPo10} with the
resulting spherical designs being available at \cite{Gr_pointsS2}.
Finally, one can try to compute both quadrature nodes and weights in an
optimization schema, cf.~\cite{DissGr13}.

For this section it is sufficient that the nodes
$\mathcal X$ and the weights
$\bm W = \diag(w_{\bm x})_{\bm x \in \mathcal X}$ form an exact
quadrature rule of degree $2N$, i.e., $\bm F\herm \bm W \bm F = \mathbf I$.
Under this assumption the diagonal entries of the hat matrix
\begin{equation*}
  \bm H = \F\left(\F\herm \bm W\F+\lambda\bm{\hat W}\right)\inv\F\herm \bm W
\end{equation*}
can by computed efficiently as it is stated in the following theorem.

\begin{theorem}
  \label{theorem:s2_diagonals}
  Let the nodes $\mathcal X$ and the weights $\bm W$ form a quadrature
  formula $Q_{\mathcal X,\bm W}$ that is exact for all spherical harmonics up
  to polynomial degree $2N$ then the diagonal entry corresponding to $\bm x$ of $\bm H$
  satisfies
  \begin{equation*}
    h_{\bm x,\bm x} = \frac{w_{\bm x}}{4\pi}\sum_{n = 0}^N \frac {2n+1}{1+\lambda \hat w_{n}}.
  \end{equation*}
\end{theorem}

\begin{proof}
  Since $\F\herm \bm W\F = \bm I$ we obtain analogously to the proof of
  Theorem~\ref{theorem:simplification_torus}
  \begin{equation*}
    \label{eq:s2_hatmatrix}
    \bm H
    = \F\diag\left(\frac 1{1+\lambda \hat w_0},\dots,\frac 1{1+\lambda \hat w_N}\right)\F\herm \bm W.
  \end{equation*}
  Looking into the diagonal entry corresponding to $x$ and using the addition theorem of
  spherical harmonics, cf.~\cite[Theorem 5.11]{Mic13}, we obtain the formula
  \begin{equation*}
    h_{\bm x,\bm x}
    = w_{\bm x}\sum_{n=0}^N \frac 1{1+\lambda \hat w_n} \sum_{k=-n}^n
    Y_{n,k}(\bm x)\overline{Y_{n,k}(\bm x)}
    = \frac {w_{\bm x}}{4\pi}\sum_{n=0}^N \frac {2n+1}{1+\lambda \hat w_n}.
  \end{equation*}
\end{proof}

\begin{corollary}
\label{corollary:s2_fast}
  For fixed $\lambda$ the ordinary cross-validation score $P(\lambda)$ and the
  generalized cross-validation score $V(\lambda)$ on the two-dimensional sphere given
  quadrature nodes and weights can be computed in $\mathcal O(N^2\log N+|\mathcal X|)$
  using Algorithm~\ref{algo:cv}.
\end{corollary}
\begin{proof}
  Due to Theorem~\ref{theorem:s2_diagonals} we can compute $h_{\bm x,\bm x}$ in linear
  time.  Using equation \eqref{eq:s2_hatmatrix} applying the hat matrix has
  the same computational cost as one multiplication with $\F$ and one with
  $\F\herm$.  Using the nonequispaced fast spherical Fourier transform (NFSFT,
  cf.~\cite{kupo02}) this can be done in $\mathcal O(N^2\log N+|\mathcal X|)$.
\end{proof}

\label{example:S2quadrature}
In order to illustrate Theorem~\ref{theorem:s2_diagonals} we consider a
quadrature rule consisting of $21\,000$ approximately equidistributed nodes
and equal weights $w_{\bm x} = 4\pi/21\,000$ that is exact up to polynomial
degree $2N = 200$, as reported in \cite{Gr_pointsS2}. Since by Theorem
\ref{theorem:s2_diagonals} the diagonal entries $h_{\bm x,\bm x}$ of the hat matrix
are constant multiples of the constant spatial weights $w_{\bm x}$ the
ordinary cross-validation score and the generalized cross-validation score
coincide for this setting. For weights in frequency domain we have chosen
$\hat w_n = (2n)^{2s}$ for $s=3$ which corresponds to a function with $3$
derivatives in $L_2(\S)$.

The test function consists of a sum of quadratic B-splines to which we added
an error of 5\% Gaussian noise for each node as one can see in
Figure~\ref{fig:quadrature_s2}, (a). This function was suggested in
\cite{HiQu15}.  We calculated $V(\lambda)$ and $P(\lambda)$ for
$\lambda\in[2^{-38},2^{-25}]$ using Algorithm~\ref{algo:cv} with the help of
the \textsc{Matlab} toolbox \textsc{MTEX}, cf.~\cite{mtex51}. Furthermore we
calculated the $L_2(\S)$-error using Parseval from the original $\fhat$ and
$\ftilde$ which are a byproduct of Algorithm~\ref{algo:cv}.

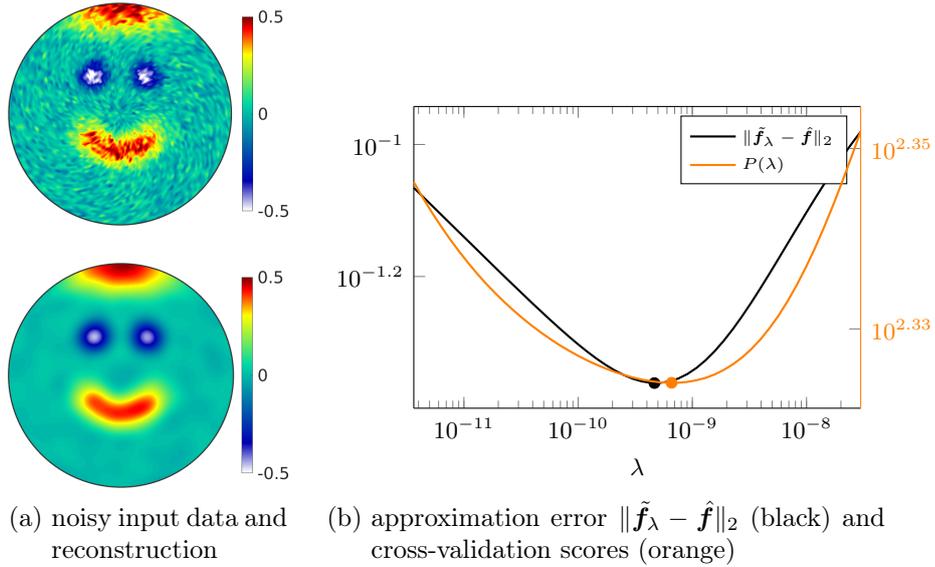
\begin{figure}
  \centering

  \begin{subfigure}[b]{0.25\textwidth}
    \includegraphics[width=\textwidth]{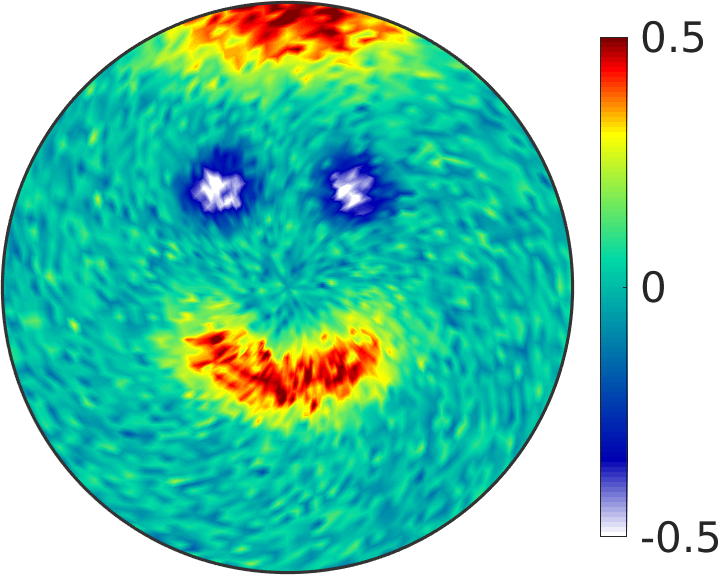}
    \\ \\
    \includegraphics[width=\linewidth]{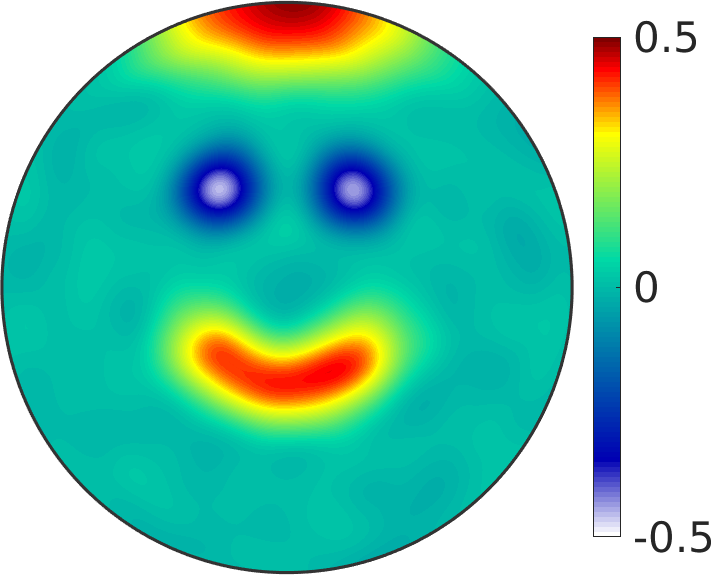}
    \caption{noisy input data and reconstruction}
  \end{subfigure}
  \quad
  \begin{subfigure}[b]{0.5\textwidth}
\begin{tikzpicture}[font = \footnotesize]
  \begin{axis}[
    scale only axis,
    width             = 0.8\textwidth, height = 4cm,
    xlabel            = $\lambda$,
    xmode             = log,
    ymode             = log,
    axis y line*      = left,
    enlarge x limits  = {abs = 0},
    legend pos        = north east,
    legend cell align = left,
    ytick             = {0.1,0.0631},legend style={font=\tiny},
  ]
    \addplot[no marks,thick] table[x index = 0, y index = 1] {data/s2_quadrature.dat};
    \addlegendentry{$\lVert \tilde \f_{\lambda} - \hat \f \rVert_{2}$};
    \addlegendimage{solid,orange,thick};
    \addlegendentry{$P(\lambda)$};
    \addplot [mark=*] coordinates {(4.6566129e-10,4.3514514e-02)};

  \end{axis}
  \begin{axis}[
    scale only axis,
    width             = 0.8\textwidth, height = 4cm,
    xmode             = log,
    ymode             = log,
    axis y line*      = right,
    axis x line       = none,
    axis line style   = {orange},
    yticklabel style  = {color = orange},
    enlarge x limits  = {abs = 0},
    ytick             = {213.7962,223.8721}]
    \addplot[no marks,orange,thick] table[x index = 0, y index = 2] {data/s2_quadrature.dat};
    \addplot [mark=*,color=orange] coordinates {(6.5854451e-10,2.1087411e+02)};
  \end{axis}
\end{tikzpicture}
    \caption{approximation error
      $\lVert \tilde \f_{\lambda} - \hat \f \rVert_{2}$ (black) and cross-validation
      scores (orange)}
\end{subfigure}

  \caption{Approximation from two-dimensional equispaced data: Comparison
    of the cross-validation score $P(\lambda)$ and the approximation error.}
\label{fig:quadrature_s2}
\end{figure}

As it is illustrated in the Figure~\ref{fig:quadrature_s2}, (b) the minimum of
the cross-validation score is very close to the minimum of the approximation
error.

\subsection{Approximative quadrature}
\label{sec:appr-quadr-S}

In the case function values are provided at nodes not forming a suitable
quadrature rule we follow the previous ideas of Section~\ref{subs:torus_app} and \ref{sec:appr-quadr}
and use the approximated hat matrix $\bm{\tilde H}$ from \eqref{eq:appr_hat} instead of $\bm H$ itself.
This way we acquire $\tilde P(\lambda)$ and $\tilde V(\lambda)$ as in Definition~\ref{def:cv_tilde}.
In place of quadrature weights we use a spherical Voronoi decomposition,
cf.~\cite{renka97}.

The only changes to Algorithm~\ref{algo:cv} are the prior computation of the
Voronoi weights and the necessity of solving a linear system of equations for
computing the Tikhonov minimizer $\fht$. This system of linear
  equation was solved with a conjugate gradient method, as implemented
  in the \textsc{Matlab} \texttt{lsqr} command, limited to 20
  iterations.

  In order to illustrate the efficiency of approximative quadrature weights
  for estimating the cross-validation score we consider the same test
  function and $\bm{\hat W}$ as in Example~\ref{example:S2quadrature} and apply Algorithm~\ref{algo:cv}
  with polynomial degree $N=30$ to $|\mathcal X|=2(N+1)^2=1922$ random nodes, which
  corresponds to an oversampling factor of two. Figure~\ref{fig:approximation_s2}
  compares the different cross-validation scores
  $P(\lambda)$, $V(\lambda)$, $\tilde P(\lambda)$ and $\tilde V(\lambda)$ for $\lambda\in[2^{-38},2^{-25}]$.
  The interval has been choosen by hand in such a way that it reflects the characteristic features of the functionals.
  All scores have their minimum close to the
  minimum of the of the actual approximation error.
  On the downside, we again observe several peaks in the
  approximated ordinary cross-validation score for small values of
  $\lambda$. So it is important to start the minimization process with a large
  $\lambda$.
  We also want to note that the computation of the exact $P(\lambda)$ and  $V(\lambda)$ took $227$ seconds averaged over all $\lambda$ in contrast to $0.12$ seconds for the approximated $\tilde P(\lambda)$ and $\tilde V(\lambda)$.

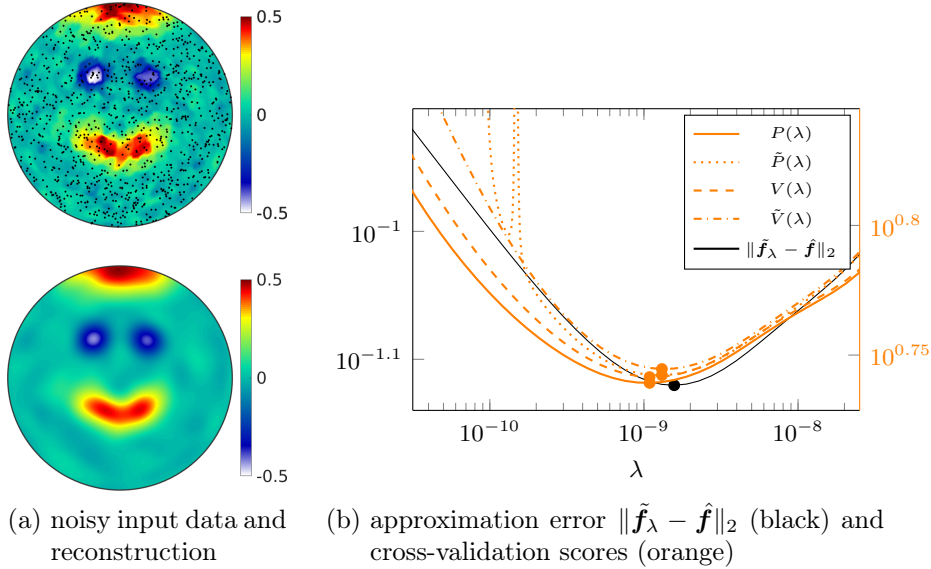
\begin{figure}
  \centering

  \begin{subfigure}[b]{0.25\textwidth}
    \includegraphics[width=\textwidth]{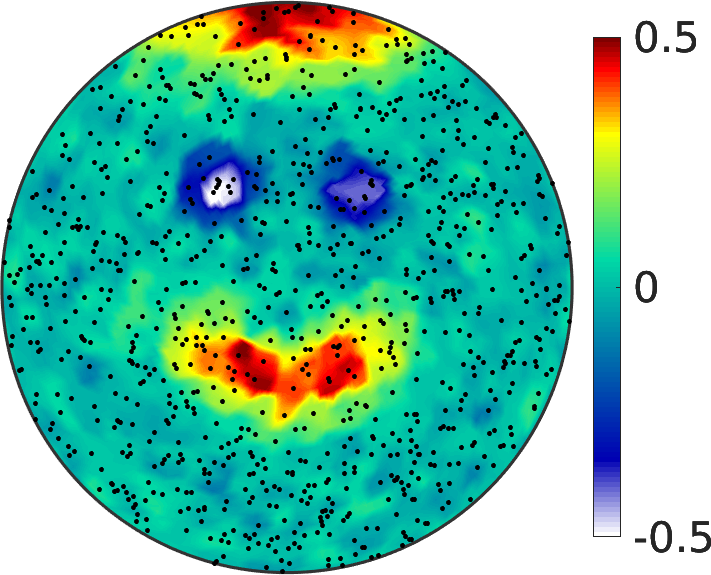}
    \\ \\
    \includegraphics[width=\linewidth]{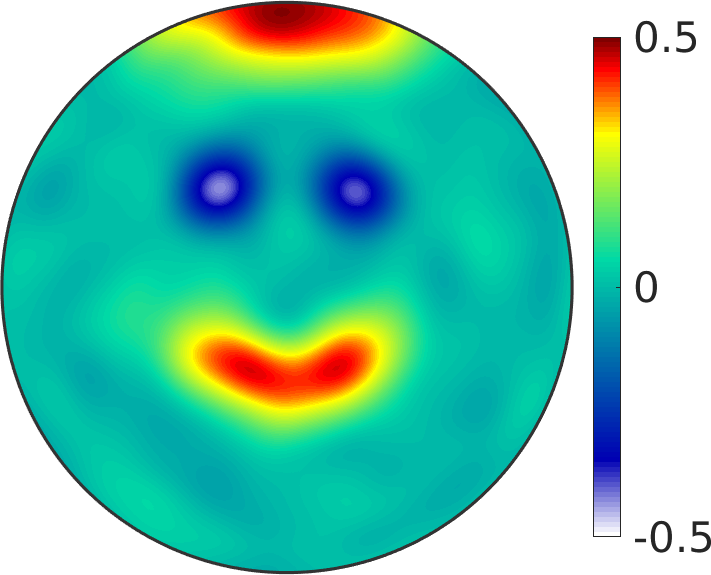}
    \caption{noisy input data and reconstruction}
  \end{subfigure}
  \quad
  \begin{subfigure}[b]{0.5\textwidth}
\begin{tikzpicture}[font = \footnotesize]
  \begin{axis}[
    scale only axis,
    width             = 0.8\textwidth, height = 4cm,
    xlabel            = $\lambda$,
    xmode             = log,
    xmin = 10^-10.5, xmax = 10^-7.6,
    ymode             = log,
    axis y line*      = left,
    enlarge x limits  = {abs = 0},
    legend pos        = north east,
    legend cell align = left,
    legend style      = {at = {(1.2,0.5)},anchor = west},
  ]
    \addplot[no marks] table[x index = 0, y index = 1] {data/s2_approximation.dat};
    \addplot [mark=*] coordinates {( 1.5718406e-09,7.5805484e-02)};

  \end{axis}
  \begin{axis}[
    scale only axis,
    width             = 0.8\textwidth, height = 4cm,
    ymax              = 7,
    xmode             = log,
    xmin = 10^-10.5, xmax = 10^-7.6,
    ymode             = log,
    axis y line*      = right,
    axis x line       = none,
    axis line style   = {orange},
    yticklabel style  = {color = orange},
    enlarge x limits  = {abs = 0},
    legend style={font=\tiny},
  ]
    \addplot[no marks,orange,thick] table[x index = 0, y index = 2] {data/s2_approximation.dat};
    \addplot[no marks,orange,dotted,thick] table[x index = 0, y index = 1] {data/s2_approximation_ocv_appr.dat};
    \addplot[no marks,orange,dashed,thick] table[x index = 0, y index = 4] {data/s2_approximation.dat};
    \addplot[no marks,orange,dashdotted,thick] table[x index = 0, y index = 5] {data/s2_approximation.dat};

    \legend{$P(\lambda)$, $\tilde P(\lambda)$, $V(\lambda)$,
      $\tilde V(\lambda)$}; \addlegendimage{solid,thick};
    \addlegendentry{$\lVert \tilde \f_{\lambda} - \hat \f \rVert_{2}$};

    \addplot [mark=*,color=orange] coordinates {( 1.0881238e-09,5.4867624e+00)};
    \addplot [mark=*,color=orange] coordinates {( 1.3078062e-09,5.5247181e+00)};
    \addplot [mark=*,color=orange] coordinates {( 1.0881238e-09,5.5154147e+00)};
    \addplot [mark=*,color=orange] coordinates {( 1.3078062e-09,5.5548687e+00)};

  \end{axis}
\end{tikzpicture}
    \caption{approximation error
      $\lVert \tilde \f_{\lambda} - \hat \f \rVert_{2}$ (black) and cross-validation
      scores (orange)}
\end{subfigure}

\caption{Approximation from two-dimensional random nodes: Comparison of the ordinary cross-validation score $P(\lambda)$ and the generalized cross-validation score $V(\lambda)$ with their approximations $\tilde P(\lambda)$ and $\tilde V(\lambda)$ and the approximation error.}
\label{fig:approximation_s2}
\end{figure}

\section{Conclusion}

In this paper we presented a fast algorithm for the computation of the
leave-one-out cross-validation score $P(\lambda)$ for the Tikhonov regularizer
\eqref{eq:tikhonov}. While many approximations of the
  cross-validation score do not consider every node separately, but fewer and
  larger validation sets, c.f.\,\cite{CY10}, we considered the full leave-one-out
  cros-validation score. In contrast to other approaches we did not
restrict ourselves to spline interpolation on the interval at equispaced nodes,
but considered more general domains and samplings. The key points of
Algorithm~\ref{algo:cv} are explicit formulas for the diagonal elements
$h_{x,x}$ of the hat matrix $\bm H$ which we were able to derive in the
Theorems \ref{theorem:simplification_torus}, \ref{theorem:i_diagonals}, and
\ref{theorem:s2_diagonals}, for approximation on the torus, the interval, and
the two-dimensional sphere, respectively. Generalizations to other domains,
e.g., the rotation group $SO(3)$, are possible following the framework
presented in this paper. For all these domains FFT-like algorithms can be
applied to achieve quasilinear complexity with respect to the number of nodes
for the computation of the Tikhonov minimizer as well as for the leave-one-out
cross-validation score.

The efficiency of our approach has been illustrated in several numerical
experiments with respect to the different domains. For the nodes we
distinguished two settings.
For nodes belonging to a quadrature rule, like equispaced nodes or rank-1 lattices on the torus, our Algorithm~\ref{algo:cv} computes the cross-validation score $P(\lambda)$ with floating point precision, cf.~Corollaries~\ref{corollary:t_fast_exact},
\ref{corollary:i_fast_exact}, and \ref{corollary:s2_fast}. For arbitrary
nodes we accomplished in Remarks~\ref{remark:approximationisfast},
\ref{remark:i_fast_appr} and Corollary~\ref{corollary:s2_fast} a good
approximation using Voronoi weights in place of the quadrature weights. The
numerical experiments confirm our theoretical results. In all test scenarios
our algorithm was several orders of magnitude faster then the direct reference
implementation.

In some cases the approximated leave-one-out cross-validation score
$\tilde P(\lambda)$ suffered from peaks for $\lambda$ smaller than the optimal
one, cf.~Subsection~\ref{subs:torus_app}. Anyway, in our test cases we had no
problems finding the global minimum by initializing the line search algorithm
with a sufficiently large $\lambda$ and thus avoiding the oscillatory region.

All relevant \textsc{Matlab} code, including the algorithm for the fast computation of the
leave-out-one cross-validation score, its minimizer and all numerical examples
of this paper can be found on the GitHub repository
\texttt{https://github.com/felixbartel/fcv}.

\bibliographystyle{abbrv}

\end{document}